\documentclass[reqno]{amsart}

\usepackage[margin=30mm]{geometry}
\usepackage{graphicx}
\usepackage{amssymb}
\usepackage{mathrsfs}
\usepackage{xcolor}
\usepackage{enumitem}
\usepackage{etoolbox}
\usepackage{natbib}
\usepackage{mleftright}
\usepackage{mathtools}
\usepackage{hyperref}
\usepackage{lipsum}

\newtheorem{theorem}{Theorem}[section]

\newtheorem{lemma}[theorem]{Lemma}
\newtheorem{proposition}[theorem]{Proposition}

\newtheorem{remark}[theorem]{Remark}

\DeclareMathOperator*{\supp}{supp}

\numberwithin{equation}{section}
\hypersetup{colorlinks=false}
\allowdisplaybreaks[2]
\mleftright

\begin{document}
\title[Temporal decays for the Navier--Stokes--Coriolis equations]{Temporal decay estimates for global solutions of the Navier--Stokes equations with the Coriolis force} 
\author[T. Yoshizawa]{Tomoaki Yoshizawa}
\date{}
\address{Graduate School of Mathematical Sciences, The University of Tokyo, 3-8-1 Komaba, Meguro-ku, Tokyo, 153-8914, Japan.}
\email{tomoaki@ms.u-tokyo.ac.jp}
\keywords{Navier--Stokes equations, Coriolis force, smoothing effect, temporal decay estimates.}
\subjclass[2020]{76D05, 76U05, 35B40.
}
\maketitle
\begin{abstract}
We consider temporal decay estimates for global solutions of the Navier--Stokes equations with the Coriolis force. We show that under several conditions including the smallness of the initial data, the solution decays as fast as the corresponding linearized solutions, and its decay rate is higher than expected from the flow of the heat equation. The estimates are derived for all $L^p$-norms with $p\in[2, \infty].$
\end{abstract} 
\tableofcontents


\section{Introduction} \label{sect:1}
In this article, we consider the Cauchy problem of the Navier--Stokes equations with the Coriolis force in
	$\mathbb R^3.$
The system is given by the equations
	\begin{align}
		\begin{dcases*}
			\partial_tu+(u\cdot \nabla)u +\nabla p-\Delta u+\Omega e_3\times u=0 & in $(0, \infty)\times\mathbb R^3$, \\
			\operatorname{div} u=0 &in $(0, \infty)\times\mathbb R^3$, \\
			u(0, \cdot)=u_0 & in $\mathbb R^3,$
		\end{dcases*}
		\label{eq:1}
	\end{align}
where the unknowns
	$u=u(t, x)\colon (0, \infty)\times\mathbb R^3\to\mathbb R^3$
and
	$p=p(t, x)\colon (0, \infty)\times\mathbb R^3\to\mathbb R$
are the velocity field and the pressure of the fluid, respectively.
The vector
	$e_3$
denotes the unit vector along the 
	$z$-axis
	$e_3\coloneqq (0, 0, 1)\in\mathbb R^3$, 
and
	$\Omega\in\mathbb R\setminus\{0\}$
is the Coriolis parameter, which describes the speed of rotation of the fluid.
Here we assume that the initial data
	$u_0$
satisfies the divergence-free condition
	$\operatorname{div}u_0=0.$
	
The aim of this article is to prove smoothing effect and temporal decay estimates for the unique global solution of the system~\eqref{eq:1} obtained by \cite{IT13} and \cite{KLT14}. More specifically, we show that under some suitable conditions on the initial data
	$u_0\in L^1(\mathbb R^3)^3\cap \dot H^s(\mathbb R^3)^3$
for
	$\frac12<s<\frac9{10}, $
the global solution
	$u(t)$
is smooth away from
	$t=0, $
and decays faster than solutions of the heat equation thanks to the dispersive effect induced by the term
	$\Omega e_3\times u$.
In particular, we obtain the faster decays in all
	$L^p$-type
norms with
	$p\in[2, \infty].$

Let us first discuss the system we consider in more detail. 
We first define the Helmholtz projection operator
	$\mathbb P$
and a skew-symmetric matrix
	$J$
by
	\begin{align*}
		\mathbb P
		\coloneqq(\delta_{jk}+R_jR_k)_{1\leq j, k\leq 3}
		\quad\text{and}\quad
		J
		\coloneqq\begin{pmatrix}
		0 & -1 & 0\\
		1 & 0 & 0\\
		0 & 0 & 0
		\end{pmatrix},
	\end{align*}
where the operator
	$R_j$
is the Riesz transform defined by the formula
	$R_j
	\coloneqq -\partial_j(-\Delta)^{-\frac12}.$
We also use
	$P(\xi)$
to denote the Fourier multiplier matrix of
	$\mathbb P$
defined by
	\begin{align*}
		P(\xi)
		\coloneqq \left(\delta_{jk}-\frac{\xi_j\xi_k}{|\xi|^2}\right)_{1\leq j, k\leq 3}.
	\end{align*}
Then applying the Helmholtz projection
	$\mathbb P$
to the first equation of \eqref{eq:1}, we may remove the pressure term
	$\nabla p$ 
to obtain
	\begin{align}
		\begin{dcases*}
			\partial_tu+\mathbb P(u\cdot \nabla)u -\Delta u+\Omega \mathbb P J \mathbb P u=0 & in $(0, \infty)\times\mathbb R^3$, \\
			u(0, \cdot)=u_0 & in $\mathbb R^3.$
		\end{dcases*}
		\label{eq:3}
	\end{align}
Regarding the global well-posedness of the system~\eqref{eq:3}, a substantial amount of work has been done.
In \cite{BMN95, BMN97, BMN99, BMN01}, Babin, Mahalov, and Nicolaenko consider the system~\eqref{eq:3} under periodic boundary conditions, and they establish the existence and regularity of the global solutions when the absolute value of the Coriolis parameter
	$|\Omega|$
is sufficiently large. Chemin, Desjardins, Gallagher, and Grenier~\cite{CDGG02, CDGG06} study the problem in the whole space
	$\mathbb R^3, $
and show that for any initial velocity
	$u_0\in L^2(\mathbb R^3)^3+ H^\frac12(\mathbb R^3)^3$
there is a constant
	$C>0$
such that if
	$|\Omega|\geq C, $
the system \eqref{eq:3} admits the unique global solution.
Hieber and Shibata~\cite{HS10} prove the global well-posedness of the system for small initial data in
	$H^\frac12(\mathbb R^3), $
where the smallness condition is given uniformly with respect to
	$\Omega\in\mathbb R.$
We also refer to \cite{GIMS08, KY11, IT14} for the global well-posedness of \eqref{eq:3} in various scaling critical function spaces (see Remark~\ref{rem:1.6} for the scaling invariance).

Our argument in this article is based on the global well-posedness result of the works of \cite{IT13} and \cite{KLT14}, and therefore let us describe their setup more precisely.
Using the linear propagator 
	$e^{t(\Delta-\Omega\mathbb PJ\mathbb P)}, $
the system~\eqref{eq:3} can be rewritten to the integral equation formally equivalent to the original problem, given by
	\begin{align}
		u(t)
		=e^{t(\Delta-\Omega\mathbb P J\mathbb P)}u_0
		-\int_0^t e^{(t-\tau)(\Delta-\Omega\mathbb P J\mathbb P)}\mathbb P\operatorname{div}(u(\tau)\otimes u(\tau))\, d\tau.
		\label{eq:4}
	\end{align}
Furthermore, as is known from~\cite{GIMM06, HS10}, 
for divergence-free initial data
	$u_0, $
the operator
	$e^{t(\Delta-\Omega\mathbb P J\mathbb P)}$
can be explicitly written by the formula
	\begin{align*}
		e^{t(\Delta-\Omega\mathbb P J\mathbb P)}u_0
		=T_\Omega(t)u_0
		&\coloneqq\mathcal F^{-1}\left(e^{-t|\xi|^2}\left(\cos\left(\Omega\frac{\xi_3}{|\xi|}t\right)I+\sin\left(\Omega\frac{\xi_3}{|\xi|}t\right)R(\xi)\right)\widehat{u_0}(\xi)\right)\\
		&=\frac12\left(
		e^{t\Delta}\mathscr G_+(\Omega t)(I+\mathcal R)u_0
		+e^{t\Delta}\mathscr G_-(\Omega t)(I-\mathcal R)u_0
		\right), 
	\end{align*}
where 
	$R(\xi)$
is the matrix defined as
	\begin{align*}
		R(\xi)
		\coloneqq \frac1{|\xi|}
			\begin{pmatrix}
				0 & \xi_3 & -\xi_2\\
				-\xi_3 & 0 & \xi_1\\
				\xi_2 & -\xi_1 & 0
			\end{pmatrix}, 
	\end{align*}
and we also define the operators
	$\mathscr G_{\pm}(\tau)$
and
	$\mathcal R$
by
	\begin{align*}
		\mathscr G_{\pm}(\tau)f
		\coloneqq \mathcal F^{-1}\big(e^{\pm i\tau\frac{\xi_3}{|\xi|}}\widehat f\big)
		\quad\text{for $\tau\in\mathbb R$, and}\quad
		\mathcal Rf
		\coloneqq \mathcal F^{-1}\big(R(\xi)\widehat f\big). 
	\end{align*}
For the integral equation~\eqref{eq:4}, the global well-posedness result of Iwabuchi and Takada~\cite{IT13}, and Koh, Lee, and Takada~\cite{KLT14} reads as follows.
\begin{theorem}[\protect{\cite[Theorem~1.1]{IT13}, \citep[Theorem~1.1]{KLT14}}]\label{thm:1}
	Let
		$s, q, $
	and
		$\theta$
	satisfy
		\begin{align}
			&\frac12
			<s
			<\frac9{10}, \quad
			\frac13+\frac s9
			\leq \frac1q
			<\frac7{12}-\frac s6\label{eq:5}
			\\
			&\frac32\left(\frac12-\frac1q\right)
			\leq\frac1\theta
			\leq\frac52\left(\frac12-\frac1q\right), 
			\quad\text{and}\quad
			\frac1{2q}+\frac s2-\frac12
			<\frac1\theta
			<\frac58-\frac3{2q}+\frac s4.\label{eq:6}
		\end{align}
	Then, there is a constant
		$C_*>0$
	such that for any
		$u_0\in \dot H^s(\mathbb R^3)^3$
	with
		$\operatorname{div}u_0=0$
	and
		$\Omega\in\mathbb R\setminus\{0\}$
	satisfying
		\begin{align*}
			|\Omega|^{-\frac12(s-\frac12)}\left\|u_0\right\|_{\dot H^s}\leq C_*, 
		\end{align*}
	the equation~\eqref{eq:4} admits a unique global solution
		$u\in C([0, \infty); \dot H^s(\mathbb R^3)^3)
		\cap L^\theta([0, \infty); \dot W^{s, q}(\mathbb R^3)^3).$
	Moreover, there is a constant 
		$C>0$
	independent of
		$u_0, $
		$u, $
	and
		$\Omega$
	satisfying
		\begin{align*}
			\left\|u\right\|_{L^\theta([0, \infty); \dot W^{s, q})}
			\leq C|\Omega|^{-\frac1\theta+\frac32(\frac12-\frac1q)}\left\|u_0\right\|_{\dot H^s}.
		\end{align*}
\end{theorem}
The result is later extended to the case of the fractional Laplacian
	$(-\Delta)^\alpha$
for
	$\frac12<\alpha<\frac52$
by Ahn, Kim, and Lee~\cite{AKL22}. They also derive a Serrin-type regularity condition and temporal decay estimates of the solution, the latter being discussed later.

Our first theorem in this article is to give another proof for the smoothing effect of the nonlinear flow of the system~\eqref{eq:3}. Specifically, we show that the global solution obtained in Theorem~\ref{thm:1} is smooth away from
	$t=0$, 
gives a classical solution of the system~\eqref{eq:3}, and satisfies the energy equality. The proof is carried out by the direct estimate of the Duhamel representation of the solution, and the strategy is based on the author's recent result~\cite{Yos26} on the viscous
	$\beta$-plane
equations.
\begin{theorem}\label{thm:4}
 Let the exponents
 	$s, $
	$q, $
and
	$\theta$
satisfy the conditions~\eqref{eq:5} and \eqref{eq:6}, 
and assume that the initial data
 	$u_0\in H^s(\mathbb R^3)^3$
satisfies
	$|\Omega|^{-\frac12(s-\frac12)}\left\|u_0\right\|_{\dot H^s}\leq C_*.$
Then the solution
	$u$
obtained in Theorem~\ref{thm:1} satisfies
	\begin{align*}
		u\in C([0, \infty); H^s(\mathbb R^3)^3)\cap C^1((0, \infty); H^m(\mathbb R^3)^3)
	\end{align*}
and
	\begin{align*}
		\partial_tu +\mathbb P(u\cdot \nabla)u-\Delta u+\Omega \mathbb P J \mathbb Pu=0
		\quad\text{in $H^m(\mathbb R^3)^3$}
	\end{align*}
for all
	$t>0$
and
	$m\geq0.$
Furthermore, the energy equality
	\begin{align*}
		\left\|u(t)\right\|_{L^2}^2
		+2\int_0^t\left\|\nabla u(\tau)\right\|_{L^2}^2\, d\tau
		=\left\|u_0\right\|_{L^2}^2
	\end{align*}
holds for all
	$t>0.$
\end{theorem}

Let us next consider the temporal decay estimates for the solution of \eqref{eq:1}. In the case of
	$\Omega=0$, 
the system coincides with the original Navier--Stokes equations 
	\begin{align}
		\begin{dcases*}
			\partial_tu+(u\cdot \nabla)u +\nabla p-\Delta u=0 & in $(0, \infty)\times\mathbb R^n$, \\
			\operatorname{div} u=0 &in $(0, \infty)\times\mathbb R^n$, \\
			u(0, \cdot)=u_0 & in $\mathbb R^n,$
		\end{dcases*}
		\label{eq:NS}
	\end{align}
in
	$n=3$, 
and the temporal decays of the solutions have been extensively studied.
In~\cite{Kat84}, Kato shows that the unique strong solution of \eqref{eq:NS} satisfies the decay estimates
	\begin{align*}
		\left\|u(t)\right\|_{L^p}
		\leq Ct^{-\frac n2(\frac1n-\frac1p)}
		\quad\text{for
	$t>0, $
	$p\in[n, \infty], $}
	\end{align*}
if the initial data
	$u_0$
is sufficiently small in
	$L^n(\mathbb R^n)^n.$
Schonbek~\cite{Sch85} shows that for the initial data
	$u_0\in L^1(\mathbb R^3)\cap L^2(\mathbb R^3)^3, $
there exists a weak solution of the system satisfying
	$\left\|u(t)\right\|_{L^2}
	\leq C(1+t)^{-\frac14}.$
The power of the decay rate is later improved by Schonbek~\cite{Sch86} and Kajikiya and Miyakawa~\cite{KM86} to
	\begin{align*}
		\left\|u(t)\right\|_{L^2}
		\leq C(1+t)^{-\frac n4}
		\quad\text{for
			$t>0$}
	\end{align*}
if 
	$u_0\in L^1(\mathbb R^n)^n\cap L^2(\mathbb R^n)^n, $
which is the same order as expected from the heat equation. Thereafter, Wiegner~\cite{Wie87} shows that if the initial data
	$u_0\in L^2(\mathbb R^n)^n$
satisfies
	$\left\|e^{t\Delta}u_0\right\|_{L^2}\leq C(1+t)^{-\alpha}$
for
	$\alpha\geq0, $
then the corresponding weak solution satisfies
	\begin{align*}
		\left\| u(t)\right\|_{L^2}\leq C(1+t)^{-\min\left\{\alpha, \frac n4+\frac12\right\}}\quad\text{for $t>0$}.
	\end{align*}
In particular, if the initial data satisfies
	$u_0\in L^2(\mathbb R^3)^3$
and
	$(1+|x|)u_0\in L^1(\mathbb R^3)^3, $
we may take
	$\alpha=\frac54$
and hence obtain
	$\left\| u(t)\right\|_{L^2}\leq C(1+t)^{-\frac54}.$
In \cite{SW96}, Schonbek and Wiegner consider the temporal decays of the higher order norms in the case of
	$n\leq 5, $
and show that if the solution satisfies
	$\left\|u(t)\right\|_{L^2}=O(t^{-\mu})$
as
	$t\to\infty$
for some
	$\mu\geq0, $
then it holds that
	$\left\|\nabla^mu(t)\right\|_{L^2}
	=O(t^{-\mu-\frac m2})$
as
	$t\to\infty$
for all
	$m\in\mathbb N$.
Fujigaki and Miyakawa~\cite{FM01} prove that if the initial data
	$u_0\in L^n(\mathbb R^n)^n$
is small and satisfies
	$(1+|x|)u_0\in L^1(\mathbb R^n)^n, $
then the unique strong solution satisfies
	$\left\|u(t)\right\|_{L^p}
	\leq Ct^{-\frac12-\frac n2(1-\frac1p)}$
for
	$t>0$
and
	$p\in[1, \infty].$
They also derive asymptotic profiles of the solution in
	$L^p(\mathbb R^n)^n$
as
	$t\to\infty$.

As for the case of
	$\Omega\neq0$, 
due to dispersion, the linearized solution exhibits faster decays
	\begin{align*}
		\left\|T_\Omega(t)u_0\right\|_{\dot W^{m, p}}
		\leq C(1+|\Omega|t)^{-1+\frac2p}t^{-\frac32(1-\frac1p)}\left\|u_0\right\|_{L^1}
		\quad\text{for $t>0,$}
	\end{align*}
where
	$p\in[2, \infty]$
and
	$m\geq0.$
Therefore, it is of interest to see whether the solution of the nonlinear problem~\eqref{eq:3} also satisfies the faster temporal decay estimates. In this direction, Ahn, Kim, and Lee~\cite{AKL22} consider the decays of the solution of \eqref{eq:4} in Theorem~\ref{thm:1}, and obtain the estimates of the form
	\begin{align*}
		\left\|u(t)\right\|_{L^p}
		\leq C\left\|u_0\right\|_{L^r}t^{-\frac32(\frac1r-\frac1p)}(1+|\Omega|t)^{-1+\frac2p}
		\quad\text{for
			$t>0$, }
	\end{align*}
where
	$C>0$
is independent of
	$u, t, $
and
	$\Omega$, 
and the exponents
	$p$
and
	$r$
satisfy
	\begin{align*}
		\left(\frac13<\right)
		\frac12-\frac12\left(\frac1q-\frac s3\right)
		\leq \frac1p
		\leq\frac1q
		\quad\text{and}\quad
		\frac1r=\frac13+\frac1q+\frac1p-\frac s3.
	\end{align*}
Kim~\cite{Kim22} considers the magnetohydrodynamics equations with the Coriolis force, and show that for initial velocity field
	$u_0\in L^1(\mathbb R^3)^3$
satisfying
	$u_0\in H^s(\mathbb R^3)^3$
for some 
	$\frac12<s<\frac32$
and several suitable conditions, the global solution
	$u$
exists and exhibits the faster decays
	\begin{align*}
		\left\|u(t)\right\|_{L^p}
		\leq Ct^{-\frac32(1-\frac1p)}(|\Omega|t)^{-1+\frac2p}
	\end{align*}
with
	$0<\gamma\leq s-\frac12$
and
	$(\frac13<)
	\max\left\{\frac14+\frac s6, \frac{4-\gamma}{8+\gamma}\right\}
	<\frac1p
	\leq \frac12.$
Later on, Egashira and Takada~\cite{ET23} also study temporal decays of the solution in Theorem~\ref{thm:1},  
and under the conditions 
	$u_0\in L^1(\mathbb R^3)^3$
or
	$(1+|x|)u_0\in L^1(\mathbb R^3)^3$, 
they obtain the decay estimates
	\begin{align*}
		\left\|u(t)\right\|_{L^p}
		\leq Ct^{-\frac32(1-\frac1p)}(1+|\Omega|t)^{-1+\frac2p}
		\quad\text{and}\quad
		\lim_{t\to\infty}t^{\frac32(1-\frac1p)}(1+|\Omega|t)^{1-\frac2p}\left\|u(t)\right\|_{L^p}
		=0, 
	\end{align*}
and
	\begin{align*}
		\left\|u(t)\right\|_{L^p}
		\leq Ct^{-\frac12-\frac32(1-\frac1p)}(1+|\Omega|t)^{-1+\frac2p}
	\end{align*}
for all
	$t>0, $
and the exponent
	$p$
satisfying
	\begin{align}
		\frac12-\frac12\left(\frac1q-\frac s3\right)
		\leq \frac1p
		\leq \frac12.\label{eq:9}
	\end{align}
They also derive the asymptotic profile of the solution in the case of
	$(1+|x|)u_0\in L^1(\mathbb R^3)^3, $
which corresponds to the result of \cite{FM01} for	
	$\Omega=0.$
We note that the range of the exponents
	$p$
permitted in \cite{AKL22, Kim22, ET23} is restricted to a bounded subinterval of
	$[2, 3)$, 
namely, an interval of the form
	$[2, p_*]$
for some
	$p_*<3.$
We also refer to~\cite{IKNP24} for the temporal decays of the scaling critical case
	$\dot H^\frac12(\mathbb R^3)^3.$
Meanwhile, the author~\cite{Yos26} has recently considered  decay estimates of the global solution for the viscous $\beta$-plane equations, and shown that if the initial data
	$\omega_0\in L^1(\mathbb R^2)$
satisfies several conditions, the corresponding global solution exhibits the faster decays
	\begin{align*}
		\left\|\omega(t)\right\|_{\dot W^{s, p}}
		\leq Ct^{-\frac s2-1+\frac1p}\min\left\{1, |\beta|^{-1+\frac2p}t^{-\frac32(1-\frac2p)}\right\}
	\end{align*}
for all
	$t>0$, 
	$s\geq 0$, 
and
	$p\in[2, \infty].$

Here we show that we may extend the range of the exponent
	$p$
from the one given by \eqref{eq:9} to
	$[2, \infty], $
and also give decay estimates of the spatial derivatives of
	$u, $
based on a strategy similar to that of \cite{Yos26}.
Recall that
	$C_*>0$
is the constant appearing in the statement of Theorem~\ref{thm:1}. Our theorem on this issue now reads: 
\begin{theorem}\label{thm:3}
	Let the exponents
		$s, q, $
	and
		$\theta$
	satisfy the conditions \eqref{eq:5} and \eqref{eq:6}.
	Then, there is a positive constant
		$C_{\star}(\leq C_*)$
	such that for any initial velocity
		$u_0\in L^1(\mathbb R^3)^3\cap \dot H^s(\mathbb R^3)^3$
	satisfying
		$\operatorname{div}u_0=0$
	and
		$|\Omega|^{-\frac12(s-\frac12)}\left\|u_0\right\|_{\dot H^s}
		\leq C_{\star}, $
	it follows that for any
		$m\geq 0$
	and
		$p\in[2, \infty], $
	there exists a positive constant
		$C$
	dependent only on
		$s, $
		$q, $
		$\theta, $
		$m, $
		$p, $
		$|\Omega|\left\|u_0\right\|_{L^1}, $
		$|\Omega|^\frac14\left\|u_0\right\|_{L^2}, $
	and
		$|\Omega|^{-\frac12(s-\frac12)}\left\|u_0\right\|_{\dot H^s}, $
	such that
	the global solution of \eqref{eq:4} obtained in Theorem~\ref{thm:1} satisfies
		\begin{align*}
			\left\|u(t)\right\|_{\dot W^{m, p}}
			\leq C|\Omega|^{-1}t^{-\frac m2-\frac32(1-\frac1p)}(1+|\Omega|t)^{-1+\frac 2p}
		\end{align*}
	for all
		$t>0, $
	and 
		\begin{align*}
			\lim_{t\to\infty}t^{\frac m2+\frac32(1-\frac1p)}(1+|\Omega|t)^{1-\frac 2p}\left\|u(t)\right\|_{\dot W^{m, p}}
			=0.
		\end{align*}
	Moreover, if
		$|x|u_0\in L^1(\mathbb R^3)^3$
	holds, then there exists a positive constant
		$C'$
	dependent additionally on
		$|\Omega|^{\frac32}\left\||x|u_0\right\|_{L^1}, $
	satisfying
		\begin{align*}
			\left\|u(t)\right\|_{\dot W^{m, p}}
			\leq C'|\Omega|^{-\frac32}t^{-\frac12-\frac m2-\frac32(1-\frac1p)}(1+|\Omega|t)^{-1+\frac2p}
		\end{align*}
	for all
		$t>0.$
\end{theorem}
\begin{remark}\label{rem:4}\rm
In the case of
	$(1+|x|)u_0\in L^1(\mathbb R^3)^3, $
it is also possible to derive asymptotic profiles of the solution of the type in \citep[Theorem~2.3]{ET23} with respect to
	$\dot W^{m, p}$-norms
for
	$m\geq 0$
and
	$p\in [2, \infty].$
\end{remark}
\begin{remark}\label{rem:1.6}\rm
For any
	$\lambda>0, $
if a pair
	$(u, \Omega)$
gives a solution of the system~\eqref{eq:3}, then so does
	$(u_\lambda, \Omega_\lambda)$
defined by
	\begin{align*}
		u_\lambda(t, x)
		\coloneqq \lambda u(\lambda^2 t, \lambda x)
		\quad\text{and}\quad
		\Omega_\lambda
		\coloneqq \lambda^2\Omega.
	\end{align*}
Note that all the quantities
	$|\Omega|\left\|u_0\right\|_{L^1}, $
	$|\Omega|^\frac32\left\||x|u_0\right\|_{L^1}, $
	$|\Omega|^\frac14\left\|u_0\right\|_{L^2}, $
	and
	$|\Omega|^{-\frac12(s-\frac12)}\left\|u_0\right\|_{\dot H^s}$, 
which appear in the statement of Theorem~\ref{thm:3}, 
satisfy the invariance under the above scaling transformation.
\end{remark}

\section{Preliminaries}\label{sect:2}
In this section, we collect several estimates employed later in our arguments.
We start by recalling the definition of the homogeneous Besov spaces 
	$\dot B^s_{p, r}(\mathbb R^3)$. 
Let $\varphi\colon \mathbb R^3\to [0, 1]$ be a smooth radial function satisfying 
	$\supp\varphi\subset\left\{3/4\leq |\xi|\leq 8/3\right\}$ and
	$\sum_{k\in\mathbb Z}\varphi(2^{-k}\xi)=1$
for all 
	$\xi\in\mathbb R^3\setminus\{0\}$. 
We define the homogeneous Littlewood--Paley projection operators 
	$P_k$ 
for 
	$k\in\mathbb Z$ 
by 
	$P_kf\coloneqq \mathcal F^{-1}(\varphi(2^{-k}\cdot)\widehat f)$, 
	$f\in\mathcal S'(\mathbb R^3).$
For 
	$s\in\mathbb R, 1\leq p\leq \infty$
and 
	$1\leq r\leq \infty$, 
we define the homogeneous Besov norm 
	$\left\|\cdot\right\|_{\dot B^s_{p, r}}$ 
 by
	$$\left\|f\right\|_{\dot B^s_{p, r}}
	\coloneqq \left\|\left(2^{sk}\left\|P_kf\right\|_{L^p}\right)_{k\in\mathbb Z}\right\|_{\ell^r}.$$
The homogeneous Besov space
	$\dot B^s_{p, r}(\mathbb R^3)$
is the set of all 
	$f\in\mathcal S'(\mathbb R^3)$
satisfying
	\begin{align*}
		\left\|f\right\|_{\dot B^s_{p, r}}<\infty
		\quad\text{and}\quad
		\big\|\mathcal F^{-1}\big(\theta(\lambda\cdot)\widehat f\big)\big\|_{L^\infty}\to0
		\text{ as $\lambda\to\infty$ for any $\theta\in\mathcal D(\mathbb R^3).$}
	\end{align*}
We also define the homogeneous Sobolev norms
	$\left\|\cdot\right\|_{\dot W^{s, p}}$
for
	$s\in\mathbb R$
and
	$p\in[1, \infty]$
by
	\begin{align*}
		\left\|f\right\|_{\dot W^{s, p}}
		\coloneqq\left\|(-\Delta)^\frac s2f\right\|_{L^p}
		=\big\|\mathcal F^{-1}(|\xi|^s\widehat f)\big\|_{L^p}.
	\end{align*}
For basic facts including the properties of the operators
	$P_k$
and the Besov spaces
	$\dot B^s_{p, r}(\mathbb R^3)$
necessary for our analysis, see ~\cite{BCD11}, for example.

We next discuss decay estimates of the linearized solutions. 

\begin{lemma}[Dispersive estimate, \protect{\citep[Lemma~2.2]{KLT14}}]\label{lem:2.1}
For any
	$p\in[2, \infty], $
there exists a constant
	$C>0$
satisfying
	\begin{align*}
		\left\|\mathscr G_\pm(\tau)P_kf\right\|_{L^p}
		\leq C2^{3k(1-\frac2p)}(1+|\tau|)^{-1+\frac2p}
		\left\|P_kf\right\|_{L^{p'}}
	\end{align*}
for all
	$\tau\in\mathbb R, $
	$k\in\mathbb Z$
and
	$f\in \mathcal S(\mathbb R^3).$
\end{lemma}
\begin{proposition}[Linear decay estimate]\label{prop:2.2}
Let
	$s\geq0, $
	$p\in[2, \infty], $
and
	$q\in[1, p']$
be given. Then there is a constant
	$C>0$
satisfying
	\begin{align*}
		\left\|e^{t(\Delta-\Omega \mathbb P J \mathbb P)}\mathbb Pf\right\|_{\dot W^{s, p}}
		\leq C t^{-\frac s2-\frac32(\frac1q-\frac1p)}(1+|\Omega|t)^{-1+\frac2p}
		\left\|f\right\|_{L^q}
	\end{align*}
for all
	$t>0, $
	$\Omega\in\mathbb R$
and
	$f\in L^q(\mathbb R^3)^3$.
\end{proposition}
\begin{remark}\label{rem:2.3}\rm
The stated estimates are almost already known in \citep[Lemma~2.5]{Kim22} and \citep[Lemma~3.2~(1)]{ET23}, except that here the vector field
	$f\in L^q(\mathbb R^3)^3$
itself is not necessarily divergence-free.
\end{remark}
\proof 
In the case of
	$s=0$
and
	$p=q, $
the assumption
	$q\leq p'\leq p$
gives
	$p=q=2.$
Then the desired estimate follows simply by the Plancherel theorem. Thus we may assume that
	$s>0$
or
	$\frac1q>\frac1p$
hold.
Thanks to the boundedness of the Helmholtz projection
	$\mathbb P$
on the Besov spaces, it is enough to show
	\begin{align*}
		\left\|T_\Omega(t)g\right\|_{\dot B^s_{p, 1}}
		\lesssim t^{-\frac s2-\frac32(\frac1q-\frac1p)}(1+|\Omega|t)^{-1+\frac2p}
		\left\|g\right\|_{\dot B^0_{q, \infty}}
	\end{align*}
for 
	$g\in \dot B^0_{q, \infty}(\mathbb R^3).$
By Lemma~\ref{lem:2.1} and the Bernstein inequality, for each
	$k\in\mathbb Z, $
we have
	\begin{align*}
		&\left\|P_ke^{t\Delta}\mathscr G_{\pm}(\Omega t)g\right\|_{L^p}
		\lesssim 2^{3k(1-\frac2p)}(1+|\Omega|t)^{-1+\frac2p}\left\|P_{k} e^{t\Delta}g\right\|_{L^{p'}}\\
		&\lesssim 2^{3k(1-\frac2p)}(1+|\Omega|t)^{-1+\frac2p}2^{3k(\frac1q-1+\frac1p)}
		\left\|P_{k} e^{t\Delta}g\right\|_{L^q}\\
		&\lesssim 2^{3k(\frac1q-\frac1p)}(1+|\Omega|t)^{-1+\frac2p}e^{-ct2^{2k}}
		\left\|P_{k} g\right\|_{L^q}, 
	\end{align*}
where
	$c>0$
is a universal constant.
Multiplying the inequality by
	$2^s$
and summing up with respect to
	$k\in\mathbb Z, $
we obtain
	\begin{align*}
		&\left\|T_\Omega(t)g\right\|_{\dot B^s_{p, 1}}
		\lesssim\sum_{k\in\mathbb Z}\sum_{\sigma\in\{\pm\}}2^{sk}\left\|P_ke^{t\Delta}\mathscr G_{\sigma}(\Omega t)(I+\sigma \mathcal R)g\right\|_{L^p}\\
		&\lesssim \sum_{k\in\mathbb Z}2^{sk+3k(\frac1q-\frac1p)}(1+|\Omega|t)^{-1+\frac2p}e^{-ct2^{2k}}\left\|P_k g\right\|_{L^q}\\
		&\leq \left(\sup_{\tau>0}\sum_{k\in\mathbb Z}
		(\tau2^{2k})^{\frac s2+\frac32(\frac1q-\frac1p)}
		e^{-c\tau2^{2k}}
		\right)
		 t^{-\frac s2-\frac32(\frac1q-\frac1p)}(1+|\Omega|t)^{-1+\frac2p}\left\|g\right\|_{\dot B^0_{q, \infty}}.
	\end{align*}
The condition
	$\frac s2+\frac32(\frac1q-\frac1p)>0$
ensures the finiteness of the supremum
	\begin{align*}
		\sup_{\tau>0}\sum_{k\in\mathbb Z}
		(\tau2^{2k})^{\frac s2+\frac32(\frac1q-\frac1p)}
		e^{-c\tau2^{2k}}
		<\infty.
	\end{align*}
This completes the proof.
\qed

\begin{proposition}\label{prop:2.4'}
Assume that 
	$u_0\in L^1(\mathbb R^3)^3$
satisfies
	$\operatorname{div}u_0=0.$
Then for any
	$m\geq0$
and
	$p\in[2, \infty], $
it holds that
	\begin{align*}
		\lim_{t\to\infty}t^{\frac m2+\frac32(1-\frac1p)}(1+|\Omega|t)^{1-\frac2p}\left\|T_\Omega(t)u_0\right\|_{\dot W^{m, p}}=0.
	\end{align*}
\end{proposition}
\proof
The case of
	$m=0$
is proved in \citep[Theorem~3.4~(1)]{ET23}, and here we prove the general case
	$m>0$
in a similar way. For
	$\Omega\in\mathbb R$
and
	$t>0$,
let
	$K_{\Omega, t}\colon \mathbb R^3\to\mathbb R^3$
be the integral kernel of the linearized equation
	\begin{align*}
		&K_{\Omega, t}(x)
		\coloneqq \mathcal F^{-1}\left(e^{-t|\xi|^2}\left(\cos\left(\Omega\frac{\xi_3}{|\xi|}t\right)I+\sin\left(\Omega\frac{\xi_3}{|\xi|}t\right)R(\xi)\right)\right)(x)\\
		&=\frac12\mathscr G_+(\Omega t)(I+\mathcal R)G_t(x)
		+\frac12\mathscr G_-(\Omega t)(I-\mathcal R)G_t(x), 
	\end{align*}
where
	$G_t$
is the Gauss kernel given by
	$G_t(x)
	\coloneqq (4\pi t)^{-\frac32}\exp(-\frac{|x|^2}{4t}).$
Since we have
	$u_0\in L^1(\mathbb R^3)^3$
and
	$\operatorname{div}u_0=0$, 
it holds that
	\begin{align*}
		\widehat{u_0}(0)
		=\int_{\mathbb R^3}u_0(x)\, dx
		=0.
	\end{align*}
Therefore, by the integral representation of the linear propagator, we have
	\begin{align*}
		\left(T_\Omega(t)u_0\right)(x)
		=\int_{\mathbb R^3}
		\left(K_{\Omega, t}(x-y)-K_{\Omega, t}(x)\right)
		u_0(y)\, dy.
	\end{align*}
Change of the variables in the definition of the integral kernel gives
	\begin{align*}
		&K_{\Omega, t}(x)
		=\mathcal F^{-1}\left(e^{-t|\xi|^2}\left(\cos\left(\Omega\frac{\xi_3}{|\xi|}t\right)I+\sin\left(\Omega\frac{\xi_3}{|\xi|}t\right)R(\xi)\right)\right)(x)\\
		&=t^{-\frac32}\mathcal F^{-1}\left(e^{-|\xi|^2}\left(\cos\left(\Omega\frac{\xi_3}{|\xi|}t\right)I+\sin\left(\Omega\frac{\xi_3}{|\xi|}t\right)R(\xi)\right)\right)(t^{-\frac12}x)\\
		&=\frac12t^{-\frac32}\mathscr G_+(\Omega t)(I+\mathcal R)G_1(t^{-\frac12}x)
		+\frac12t^{-\frac32}\mathscr G_-(\Omega t)(I-\mathcal R)G_1(t^{-\frac12}x), 
	\end{align*}
and hence we obtain
	\begin{align*}
		&\left\|T_\Omega(t)u_0\right\|_{\dot W^{m, p}}
		\leq \int_{\mathbb R^3}
		\left\|K_{\Omega, t}(\cdot-y)-K_{\Omega, t}\right\|_{\dot W^{m, p}}
		|u_0(y)|\, dy\\
		&\leq \frac12t^{-\frac32}\sum_{\sigma\in\{\pm\}}
		\int_{\mathbb R^3}
		\left\|\mathscr G_\sigma(\Omega t)(I+\sigma\mathcal R)G_1(t^{-\frac12}(\cdot-y))-\mathscr G_\sigma(\Omega t)(I+\sigma\mathcal R)G_1(t^{-\frac12}\cdot)\right\|_{\dot W^{m, p}}
		|u_0(y)|\, dy\\
		&\leq\frac12t^{-\frac m2-\frac32(1-\frac1p)}\sum_{\sigma\in\{\pm\}}
		\int_{\mathbb R^3}
		\left\|\mathscr G_\sigma(\Omega t)(I+\sigma\mathcal R)G_1(\cdot-t^{-\frac12}y)-\mathscr G_\sigma(\Omega t)(I+\sigma\mathcal R)G_1\right\|_{\dot B^{m}_{p, 1}}
		|u_0(y)|\, dy\\
		&\lesssim t^{-\frac m2-\frac32(1-\frac1p)}(1+|\Omega|t)^{-1+\frac2p}
		\int_{\mathbb R^3}
		\left\|G_1(\cdot-t^{-\frac12}y)-G_1\right\|_{\dot B^{m+3(1-\frac2p)}_{p', 1}}
		|u_0(y)|\, dy.
	\end{align*}
By the embeddings
	$ B^{2\lfloor \frac{m+1}2\rfloor+4}_{p', \infty}(\mathbb R^3)
	\hookrightarrow B^{m+3(1-\frac2p)}_{p', 1}(\mathbb R^3)
	\hookrightarrow\dot B^{m+3(1-\frac2p)}_{p', 1}(\mathbb R^3), $
we have
	\begin{align*}
		&\left\|G_1(\cdot-t^{-\frac12}y)-G_1\right\|_{\dot B^{m+3(1-\frac2p)}_{p', 1}}
		\lesssim \left\|G_1(\cdot-t^{-\frac12}y)-G_1\right\|_{B^{2\lfloor \frac{m+1}2\rfloor+4}_{p', \infty}}\\
		&\lesssim \left\|(I-\Delta)^{\lfloor \frac{m+1}2\rfloor+2}G_1(\cdot-t^{-\frac12}y)-(I-\Delta)^{\lfloor \frac{m+1}2\rfloor+2}G_1\right\|_{L^{p'}}, 
	\end{align*}
and for each
	$y\in\mathbb R^3, $
the bound tends to zero as
	$t\to\infty.$
Thus the dominated convergence theorem yields
	\begin{align*}
		&\lim_{t\to\infty}t^{\frac m2+\frac32(1-\frac1p)}(1+|\Omega|t)^{1-\frac2p}\left\|T_\Omega(t)u_0\right\|_{\dot W^{m, p}}\\
		&\lesssim\lim_{t\to\infty}\int_{\mathbb R^3}
		\left\|G_1(\cdot-t^{-\frac12}y)-G_1\right\|_{\dot B^{m+3(1-\frac2p)}_{p', 1}}
		|u_0(y)|\, dy
		=0.
	\end{align*}
This completes the proof.
\qed

\begin{proposition}[\protect{\citep[Lemma~3.2~(2)]{ET23}}]\label{prop:2.4}
	Let
		$s\geq 0$
	and
		$p\in[2, \infty]$
	be given. Then there is a constant
		$C>0$
	such that for any vector field
		$u_0\in L^1(\mathbb R^3)^3$
	with
		$\operatorname{div}u_0=0$
	and
		$|x|u_0\in L^1(\mathbb R^3)^3$
	it holds that
		\begin{align*}
			\left\|T_\Omega(t)u_0\right\|_{\dot W^{s, p}}
			\leq Ct^{-\frac12-\frac s2-\frac32(1-\frac1p)}
			(1+|\Omega|t)^{-1+\frac2p}\left\||x|u_0\right\|_{L^1}
		\end{align*}
	for all
		$t>0$
	and
		$\Omega\in\mathbb R.$
\end{proposition}

We also recall the fractional Leibniz rule, which is used to estimate the nonlinear term.
	\begin{proposition}[Fractional Leibniz rule, \protect{\citep[Theorem~1]{GO14}}]\label{prop:2.5}
		Let 
			$s\geq 0$ 
		and 
			$1<r<\infty$
		be given. 
		Assume also that 
			$1<p_1, p_2, q_1, q_2\leq\infty$
		satisfy
			$\frac1r=\frac{1}{p_j}+\frac{1}{q_j}$
		for
			$j=1, 2$.
		Then there exists a constant
			$C>0$
		such that
			\begin{align*}
				\left\|fg\right\|_{\dot W^{s, r}}
				\leq C\left(\left\|f\right\|_{\dot W^{s, p_1}}\left\|g\right\|_{L^{q_1}}+\left\|f\right\|_{L^{q_2}}\left\|g\right\|_{\dot W^{s, p_2}}\right)
			\end{align*}
		for all
			$f\in \dot W^{s, p_1}(\mathbb R^3)\cap L^{q_2}(\mathbb R^3)$
		and
			$g\in \dot W^{s, p_2}(\mathbb R^3)\cap L^{q_1}(\mathbb R^3)$.
	\end{proposition}

\section{Proof of Theorem~\ref{thm:4}: smoothing effect}
This section is devoted to the proof of smoothing effect of the system~\eqref{eq:1} and the derivation of the energy equality
	\begin{align*}
		\left\|u(t)\right\|_{L^2}^2
		+2\int_0^t\left\|\nabla u(\tau)\right\|_{L^2}^2\, d\tau
		=\left\|u_0\right\|_{L^2}^2
	\end{align*}
for
	$t>0.$
In the following arguments, we assume that the exponents
	$s, q, \theta$
always satisfy the conditions \eqref{eq:5} and \eqref{eq:6}, and the divergence-free initial velocity field
	$u_0$
satisfies
	$u_0\in H^s(\mathbb R^3)^3$
(i.e., 
	$u_0$
is an element of the inhomogeneous Sobolev space)
and
	$|\Omega|^{-\frac12(s-\frac12)}\left\|u_0\right\|_{\dot H^s}
	\leq C_*.$
Let 
	$u\in C([0, \infty); \dot H^s(\mathbb R^3)^3)\cap L^\theta([0, \infty); \dot W^{s, q}(\mathbb R^3)^3)$
denote the solution of the equation~\eqref{eq:4} obtained in Theorem~\ref{thm:1}.
We first observe the smoothing effect in terms of the space-time norms of the soluiton.
\begin{lemma}\label{lem:5}
	For any
		$m\geq0$
	and
		$\varepsilon>0, $
	the solution
		$u$
	satisfies
		$u \in L^\theta([\varepsilon, \infty); \dot W^{s+m, q}(\mathbb R^3)^3)$
	and
		$\left\|u\right\|_{L^\theta([\varepsilon, \infty); \dot W^{s+m, q})}
		=O(\varepsilon^{-\frac m2})$
	as
		$\varepsilon\to0.$
	Moreover, if
		$u_0\in \dot H^{s+m}(\mathbb R^3)^3$
	for
		$m\geq 0, $
	then we have
		$u \in L^\theta([0, \infty); \dot W^{s+m, q}(\mathbb R^3)^3).$
\end{lemma}
\proof
We only have to deal with the case of
	$m>0.$
Since 
	$u$
is a solution of the integral equation~\eqref{eq:4}, 
we may write
	\begin{align*}
		u(t)
		=T_\Omega(t)u_0
		-\int_0^\frac t2T_\Omega(t-\tau)\mathbb P\operatorname{div}(u(\tau)\otimes u(\tau))\, d\tau
		-\int_\frac t2^tT_\Omega(t-\tau)\mathbb P\operatorname{div}(u(\tau)\otimes u(\tau))\, d\tau
	\end{align*}
and deal with the estimates of each term on the right-hand side.

For any
	$t\geq \varepsilon, $
due to the smoothing property of the heat kernel, we have
	\begin{align*}
		\left\|T_{\Omega}(t)u_0\right\|_{\dot W^{s+m, q}}
		\lesssim t^{-\frac m2}\left\|T_{2\Omega}\left(t/2\right)u_0\right\|_{\dot W^{s, q}}, 
	\end{align*}
and therefore
	\begin{align*}
		\left\|T_{\Omega}(t)u_0\right\|_{L^\theta([\varepsilon, \infty); \dot W^{s+m, q})}
		\lesssim \varepsilon^{-\frac m2}\left\|T_{2\Omega}(t/2)u_0\right\|_{L^\theta([0, \infty); \dot W^{s, q})}<\infty.
	\end{align*}
Here we also obtain the order estimate
	$\left\|T_{\Omega}(t)u_0\right\|_{L^\theta([\varepsilon, \infty); \dot W^{s+m, q})}
	=O(\varepsilon^{-\frac m2})$
as
	$\varepsilon\to0.$
For the first half of the Duhamel term, setting
	$\frac1r\coloneqq\frac2q-\frac s3$
and using Proposition~\ref{prop:2.2}, the smoothing property similarly gives
	\begin{align*}
		&\left\|\int_0^\frac t2T_\Omega(t-\tau)\mathbb P\operatorname{div}(u(\tau)\otimes u(\tau))\, d\tau\right\|_{\dot W^{s+m, q}}
		\leq \int_0^\frac t2\left\|T_\Omega(t-\tau)\mathbb P\operatorname{div}(u(\tau)\otimes u(\tau))\right\|_{\dot W^{s+m, q}}\, d\tau\\
		&\lesssim \int_0^\frac t2(1+|\Omega|(t-\tau))^{-1+\frac 2q}(t-\tau)^{-\frac m2-\frac12-\frac32(1-\frac2q)}\left\|e^{\frac12(t-\tau)\Delta}(u(\tau)\otimes u(\tau))\right\|_{\dot W^{s, q'}}\, d\tau\\
		&\lesssim \int_0^\frac t2(1+|\Omega|(t-\tau))^{-1+\frac 2q}(t-\tau)^{-\frac m2-\frac12-\frac32(\frac1q-\frac s3)}\left\|u(\tau)\otimes u(\tau)\right\|_{\dot W^{s, r}}\, d\tau\\
		&\lesssim \int_0^\frac t2(1+|\Omega|(t-\tau))^{-1+\frac 2q}(t-\tau)^{-\frac m2-\frac12-\frac32(\frac1q-\frac s3)}
		\left\|u(\tau)\right\|_{\dot W^{s, q}}^2\, d\tau\\
		&\lesssim \varepsilon^{-\frac m2}\int_0^\frac t2(1+|\Omega|(t-\tau))^{-1+\frac 2q}(t-\tau)^{-\frac12-\frac32(\frac1q-\frac s3)}
		\left\|u(\tau)\right\|_{\dot W^{s, q}}^2\, d\tau
	\end{align*}
Since the assumption~\eqref{eq:5} and \eqref{eq:6} ensures the condition
	\begin{align*}
		\frac12-\frac32\left(\frac1q-\frac s3\right)
		-\left(1-\frac2q\right)
		<\frac1\theta
		<\frac12-\frac32\left(\frac1q-\frac s3\right), 
	\end{align*}
Young's inequality yields
	\begin{align*}
		&\left\|\int_0^\frac t2T_\Omega(t-\tau)\mathbb P\operatorname{div}(u(\tau)\otimes u(\tau))\, d\tau\right\|_{L^\theta([\varepsilon, \infty); \dot W^{s+m, q})}\\
		&\lesssim \varepsilon^{-\frac m2}|\Omega|^{\frac12+\frac32(\frac1q-\frac s3)-1+\frac 1\theta}\left\|u\right\|_{L^\theta([0, \infty); \dot W^{s, q})}^2<\infty.
	\end{align*}
Note that the order of the bound is again
	$O(\varepsilon^{-\frac m2})$
as
	$\varepsilon\to0.$
	
We now deal with the second half of the Duhamel term. The proof is based on the inductive argument employed in \cite{Yos26}. Let 
	$N\in\mathbb N $
be a sufficiently large number such that
	$\alpha\coloneqq \frac mN$
satisfies
	\begin{align*}
		\frac12-\frac32\left(\frac1q-\frac s3\right)-\left(1-\frac2q\right)-\frac\alpha2
		<\frac1\theta
		<\frac12-\frac32\left(\frac1q-\frac s3\right)-\frac\alpha2.
	\end{align*}
Here we prove that
	\begin{align*}
		&\left\|\int_\frac t2^tT_\Omega(t-\tau)\mathbb P\operatorname{div}(u(\tau)\otimes u(\tau))\, d\tau\right\|_{L^\theta([\varepsilon, \infty); \dot W^{s+k\alpha, q})}
	\end{align*}
is finite for
	$k\in \{0, 1, \dots, N\}$
and its order is
	$o(\varepsilon^{-\frac m2})$
as
	$\varepsilon\to 0$
by induction of
	$k.$
The base case
	$k=0$
follows from Theorem~\ref{thm:1}. For
	$k\geq1, $
we have
	\begin{align*}
		&\left\|\int_\frac t2^tT_\Omega(t-\tau)\mathbb P\operatorname{div}(u(\tau)\otimes u(\tau))\, d\tau\right\|_{\dot W^{s+k\alpha, q}}\\
		&\leq\int_\frac t2^t\left\|T_\Omega(t-\tau)\mathbb P\operatorname{div}(u(\tau)\otimes u(\tau))\right\|_{\dot W^{s+k\alpha, q}}\, d\tau\\
		&\lesssim \int_\frac t2^t(1+|\Omega|(t-\tau))^{-1+\frac2q}(t-\tau)^{-\frac12-\frac\alpha2-\frac32(\frac1q-\frac s3)}\left\|u(\tau)\otimes u(\tau)\right\|_{\dot W^{s+(k-1)\alpha, r}}\, d\tau\\
		&\lesssim \int_\frac t2^t(1+|\Omega|(t-\tau))^{-1+\frac2q}(t-\tau)^{-\frac12-\frac\alpha2-\frac32(\frac1q-\frac s3)}\left\|u(\tau)\right\|_{\dot W^{s+(k-1)\alpha, q}}\left\|u(\tau)\right\|_{\dot W^{s, q}}\, d\tau.
	\end{align*}
Taking the
	$L^\theta$-norm
of the estimate on
	$[\varepsilon, \infty)$
and applying Young's inequality, we subsequently obtain
	\begin{align*}
		&\left\|\int_\frac t2^tT_\Omega(t-\tau)\mathbb P\operatorname{div}(u(\tau)\otimes u(\tau))\, d\tau\right\|_{L^\theta([\varepsilon, \infty); \dot W^{s+k\alpha, q})}\\
		&\lesssim |\Omega|^{\frac12+\frac\alpha2+\frac32(\frac1q-\frac s3)-1+\frac1\theta}
		\left\|u\right\|_{L^\theta([0, \infty); \dot W^{s, q})}
		\left\|u\right\|_{L^\theta([\frac\varepsilon2, \infty); \dot W^{s+(k-1)\alpha, q})}.
	\end{align*}
By the hypothesis of the induction, the right-hand side is finite and its order is
	$o(\varepsilon^{-\frac m2})$
as
	$\varepsilon\to0.$
Consequently, by induction
for any
	$\varepsilon>0$
and
	$m\geq 0, $
we obtain
	$u\in L^\theta([\varepsilon, \infty); \dot W^{s+m, q}(\mathbb R^3)^3), $
and
	\begin{align*}
		\left\|u\right\|_{L^\theta([\varepsilon, \infty); \dot W^{s+m, q})}
		=O(\varepsilon^{-\frac{m}2})+o(\varepsilon^{-\frac{m}2})
		=O(\varepsilon^{-\frac{m}2})
		\quad\text{as $\varepsilon\to0$}.
	\end{align*}
This proves the first half of the statement.
If
	$u_0\in \dot H^{s}(\mathbb R^3)^3\cap\dot H^{s+m}(\mathbb R^3)^3, $
we can estimate the linear term as
	\begin{align*}
		\left\|T_\Omega(t)u_0\right\|_{L^\theta([0, \infty); \dot W^{s+m, q})}
		\lesssim |\Omega|^{-\frac1\theta+\frac32(\frac12-\frac1q)}\left\|u_0\right\|_{\dot H^{s+m}}<\infty, 
	\end{align*}
and it is possible to replace the estimates of
	$\displaystyle\int_\frac t2^t$
with those of
	$\displaystyle\int_0^t$
in the inductive process above to obtain
	$u\in L^\theta([0, \infty); \dot W^{s+m, q}(\mathbb R^3)^3).$
This completes the proof.
\qed
\begin{lemma}\label{lem:6}
	For any
		$m\geq0, $
	the solution
		$u$
	satisfies
		$u\in C((0, \infty); \dot H^m(\mathbb R^3)^3)$
	and
		$\left\|u(t)\right\|_{\dot H^m}
		=O(t^{-\frac m2})$
	as
		$t\to 0.$
\end{lemma}
\proof
We first show that
	$u(t)\in \dot H^{m}(\mathbb R^3)^3$
for any
	$t>0$
and
	$\left\|u(t)\right\|_{\dot H^m}
	=O(t^{-\frac m2})$
as
	$t\to0.$
We may write
	\begin{align*}
		\left\|u(t)\right\|_{\dot H^m}
		&\leq \left\|T_\Omega(t)u_0\right\|_{\dot H^m}
		+\int_0^\frac t2\left\|T_\Omega(t-\tau)\mathbb P\operatorname{div}(u(\tau)\otimes u(\tau))\right\|_{\dot H^m}\, d\tau\\
		&+\int_\frac t2^t\left\|T_\Omega(t-\tau)\mathbb P\operatorname{div}(u(\tau)\otimes u(\tau))\right\|_{\dot H^m}\, d\tau
	\end{align*}
and deal with the estimates of each term on the right-hand side.

For the first term, smoothing effect of the heat kernel gives
	\begin{align*}
		\left\|T_\Omega(t)u_0\right\|_{\dot H^m}
		\lesssim t^{-\frac m2}\left\|u_0\right\|_{L^2}<\infty.
	\end{align*}	
	
For the second term, using the smoothing effect of the heat kernel and Proposition~\ref{prop:2.5}, we have
	\begin{align*}
		&\int_0^\frac t2\left\|T_\Omega(t-\tau)\mathbb P\operatorname{div}(u(\tau)\otimes u(\tau))\right\|_{\dot H^m}\, d\tau\\
		&\lesssim \int_0^\frac t2(t-\tau)^{-\frac m2-\frac{1-s}2-\frac32(\frac2q-\frac s3-\frac12)}\left\|u(\tau)\otimes u(\tau)\right\|_{\dot W^{s, r}}\, d\tau\\
		&\lesssim t^{-\frac m2-\frac{1-s}2-\frac32(\frac2q-\frac s3-\frac12)}\int_0^\frac t2\left\|u(\tau)\right\|_{\dot W^{s, q}}^2\, d\tau
		\\
		&\lesssim t^{-\frac m2+1-\frac2\theta-\frac{1-s}2-\frac32(\frac2q-\frac s3-\frac12)}\left\|u\right\|_{L^\theta([0, \infty); \dot W^{s, q})}^2<\infty.
	\end{align*}
Note that we have
	$t^{1-\frac2\theta-\frac{1-s}2-\frac32(\frac2q-\frac s3-\frac12)}
	=t^{\frac s2}\cdot t^{2(-\frac1\theta+\frac58-\frac3{2q}+\frac s4)}, $
and by \eqref{eq:6} the bound is
	$o(t^{-\frac m2})$
as
	$t\to0.$

For the last term, we similarly have
	\begin{align*}
		&\int_\frac t2^t\left\|T_\Omega(t-\tau)\mathbb P\operatorname{div}(u(\tau)\otimes u(\tau))\right\|_{\dot H^m}\, d\tau\\
		&\lesssim \int_\frac t2^t(t-\tau)^{-\frac{1-s}2-\frac32(\frac2q-\frac s3-\frac12)}\left\|u(\tau)\otimes u(\tau)\right\|_{\dot W^{s+m, r}}\, d\tau\\
		&\lesssim \int_\frac t2^t(t-\tau)^{-\frac{1-s}2-\frac32(\frac2q-\frac s3-\frac12)}\left\|u(\tau)\right\|_{\dot W^{s, q}}\left\|u(\tau)\right\|_{\dot W^{s+m, q}}\, d\tau\\
		&\lesssim t^{1-\frac2\theta-\frac{1-s}2-\frac32(\frac2q-\frac s3-\frac12)}
		\left\|u\right\|_{L^\theta([0, \infty); \dot W^{s, q})}
		\left\|u\right\|_{L^\theta([\frac t2, \infty); \dot W^{s+m, q})}
		<\infty.
	\end{align*}
The bound here is also
	$o(t^{-\frac m2})$
as
	$t\to0, $
by Lemma~\ref{lem:5}.
Combining the above estimates, we see that
	$u(t)\in \dot H^{m}(\mathbb R^3)^3$
for all
	$t>0$
and
	$\left\|u(t)\right\|_{\dot H^m}
	=O(t^{-\frac m2})$
as
	$t\to0.$
The first half of the statement is therefore done, and now we are left to prove the continuity of
	$u(t)\in \dot H^m(\mathbb R^3)^3$
for
	$t>0.$

Let
	$\varepsilon, s, t, $
and
	$M$
satisfy
	$0<\varepsilon<\frac s2<\frac t2<s<t<M.$
We write
	\begin{align}
		&u(t)-u(s)
		=T_\Omega(t)u_0-T_\Omega(s)u_0
		\notag
		\\
		&-\int_0^\frac t2e^{(t-\tau)\Delta}e^{-(t-\tau)\Omega\mathbb PJ\mathbb P}\mathbb P\operatorname{div}(u(\tau)\otimes u(\tau))\, d\tau
		+\int_0^\frac s2e^{(t-\tau)\Delta}e^{-(t-\tau)\Omega\mathbb PJ\mathbb P}\mathbb P\operatorname{div}(u(\tau)\otimes u(\tau))\, d\tau
		\label{eq:52}\\
		&-\int_0^\frac s2e^{(t-\tau)\Delta}e^{-(t-\tau)\Omega\mathbb PJ\mathbb P}\mathbb P\operatorname{div}(u(\tau)\otimes u(\tau))\, d\tau
		+\int_0^\frac s2e^{(s-\tau)\Delta}e^{-(t-\tau)\Omega\mathbb PJ\mathbb P}\mathbb P\operatorname{div}(u(\tau)\otimes u(\tau))\, d\tau
		\label{eq:53}\\
		&-\int_0^\frac s2e^{(s-\tau)\Delta}e^{-(t-\tau)\Omega\mathbb PJ\mathbb P}\mathbb P\operatorname{div}(u(\tau)\otimes u(\tau))\, d\tau
		+\int_0^\frac s2e^{(s-\tau)\Delta}e^{-(s-\tau)\Omega\mathbb PJ\mathbb P}\mathbb P\operatorname{div}(u(\tau)\otimes u(\tau))\, d\tau
		\label{eq:54}\\
		&-\int_\frac t2^te^{(t-\tau)\Delta}e^{-(t-\tau)\Omega\mathbb PJ\mathbb P}\mathbb P\operatorname{div}(u(\tau)\otimes u(\tau))\, d\tau
		+\int^s_\frac t2e^{(t-\tau)\Delta}e^{-(t-\tau)\Omega\mathbb PJ\mathbb P}\mathbb P\operatorname{div}(u(\tau)\otimes u(\tau))\, d\tau
		\label{eq:55}\\
		&-\int^s_\frac t2e^{(t-\tau)\Delta}e^{-(t-\tau)\Omega\mathbb PJ\mathbb P}\mathbb P\operatorname{div}(u(\tau)\otimes u(\tau))\, d\tau
		+\int^s_\frac s2e^{(t-\tau)\Delta}e^{-(t-\tau)\Omega\mathbb PJ\mathbb P}\mathbb P\operatorname{div}(u(\tau)\otimes u(\tau))\, d\tau
		\label{eq:56}\\
		&-\int^s_\frac s2e^{(t-\tau)\Delta}e^{-(t-\tau)\Omega\mathbb PJ\mathbb P}\mathbb P\operatorname{div}(u(\tau)\otimes u(\tau))\, d\tau
		+\int^s_\frac s2e^{(s-\tau)\Delta}e^{-(t-\tau)\Omega\mathbb PJ\mathbb P}\mathbb P\operatorname{div}(u(\tau)\otimes u(\tau))\, d\tau
		\label{eq:57}\\
		&-\int^s_\frac s2e^{(s-\tau)\Delta}e^{-(t-\tau)\Omega\mathbb PJ\mathbb P}\mathbb P\operatorname{div}(u(\tau)\otimes u(\tau))\, d\tau
		+\int^s_\frac s2e^{(s-\tau)\Delta}e^{-(s-\tau)\Omega\mathbb PJ\mathbb P}\mathbb P\operatorname{div}(u(\tau)\otimes u(\tau))\, d\tau
		\label{eq:58}
	\end{align}
and show that the right-hand side of the equality tends to zero as
	$t-s\to0$
in
	$\dot H^m(\mathbb R^3)^3.$
Note that the sum of \eqref{eq:52} and \eqref{eq:56} is equal to 0.

For the linear part, we have
	\begin{align*}
		&\left\|T_\Omega(t)u_0-T_\Omega(s)u_0\right\|_{\dot H^m}\\
		&=\left\|(T_\Omega(t-s)-I)T_\Omega(s)u_0\right\|_{\dot H^m}\\
		&\lesssim \left\|(e^{(t-s)\Delta}-I)e^{s\Delta}u_0\right\|_{\dot H^m}
		+\left\|(e^{-(t-s)\Omega\mathbb P J\mathbb P}-I)e^{s\Delta}u_0\right\|_{\dot H^m}\\
		&\leq\left\|(e^{(t-s)\Delta}-I)e^{\varepsilon\Delta}u_0\right\|_{\dot H^m}
		+\left\|(e^{-(t-s)\Omega\mathbb P J\mathbb P}-I)e^{\varepsilon\Delta}u_0\right\|_{\dot H^m}.
	\end{align*}
Since the smoothing property of the heat kernel ensures that
	$e^{\varepsilon\Delta}u_0\in \dot H^m(\mathbb R^3)^3 , $
the dominated convergence theorem is enough to see that the right-hand side tends to zero as
	$t-s\to0.$


For \eqref{eq:53}, we have
	\begin{align*}
		&\left\|\int_0^\frac s2(e^{(t-s)\Delta}-I)e^{(s-\tau)\Delta}e^{-(t-\tau)\Omega\mathbb PJ\mathbb P}\mathbb P\operatorname{div}(u(\tau)\otimes u(\tau))\, d\tau\right\|_{\dot H^m}\\
		&\lesssim \int_0^\frac s2\left\|(e^{(t-s)\Delta}-I)e^{(s-\tau)\Delta}\operatorname{div}(u(\tau)\otimes u(\tau))\right\|_{\dot H^m}\, d\tau.
	\end{align*}
By the Plancherel theorem and the mean value theorem, we can calculate the integrand as
	\begin{align*}
		&\left\|(e^{(t-s)\Delta}-I)e^{(s-\tau)\Delta}\operatorname{div}(u(\tau)\otimes u(\tau))\right\|_{\dot H^m}\\
		&=\left(\int_{\mathbb R^3}
		\left|(t-s)|\xi|^2\int_0^1e^{-\theta(t-s)|\xi|^2}\, d\theta\right|^2
		\left|\mathcal F\left((-\Delta)^{\frac m2}e^{(s-\tau)\Delta}\operatorname{div}(u(\tau)\otimes u(\tau))\right)(\xi)\right|^2\, d\xi\right)^\frac12\\
		&\leq (t-s)\left\|e^{(s-\tau)\Delta}\operatorname{div}(u(\tau)\otimes u(\tau))\right\|_{\dot H^{m+2}}\\
		&\lesssim (t-s)(s-\tau)^{-\frac{m+2}2-\frac{1-s}2-\frac32(\frac 2q-\frac s3-\frac12)}
		\left\|u(\tau)\right\|_{\dot W^{s, q}}^2, 
	\end{align*}
and thus we obtain
	\begin{align*}
		&\left\|\int_0^\frac s2(e^{(t-s)\Delta}-I)e^{(s-\tau)\Delta}e^{-(t-\tau)\Omega\mathbb PJ\mathbb P}\mathbb P\operatorname{div}(u(\tau)\otimes u(\tau))\, d\tau\right\|_{\dot H^m}\\
		&\lesssim (t-s)\varepsilon^{-\frac{m+2}2}M^{1-\frac2\theta-\frac{1-s}2-\frac32(\frac2q-\frac s3-\frac12)}
		\left\|u\right\|_{L^\theta([0, \infty); \dot W^{s, q})}^2
		\to 0\quad\text{as $t-s\to0$.}
	\end{align*}

The estimate of \eqref{eq:54} can be done in a similar way. We have
	\begin{align}
		&\left\|\int_0^\frac s2(e^{-(t-s)\Omega\mathbb PJ\mathbb P}-I)e^{(s-\tau)\Delta}e^{-(s-\tau)\Omega\mathbb PJ\mathbb P}\mathbb P\operatorname{div}(u(\tau)\otimes u(\tau))\, d\tau\right\|_{\dot H^m}\notag\\
		&\lesssim\int_0^\frac s2\left\|(e^{-(t-s)\Omega\mathbb PJ\mathbb P}-I)e^{(s-\tau)\Delta}\mathbb P\operatorname{div}(u(\tau)\otimes u(\tau))\right\|_{\dot H^m}\, d\tau.\label{eq:71}
	\end{align}
In this case, the integrand can be calculated as
	\begin{align*}
		&\left\|(e^{-(t-s)\Omega\mathbb PJ\mathbb P}-I)e^{(s-\tau)\Delta}\mathbb P\operatorname{div}(u(\tau)\otimes u(\tau))\right\|_{\dot H^m}\\
		&=\left(\int_{\mathbb R^3}
		\left|\left(\left(\cos\left(\Omega\frac{\xi_3}{|\xi|}(t-s)\right)-1\right)I+\sin\left(\Omega\frac{\xi_3}{|\xi|}(t-s)\right)R(\xi)\right)
		\mathcal F\left(e^{(s-\tau)\Delta}\mathbb P\operatorname{div}(u(\tau)\otimes u(\tau))\right)(\xi)\right|^2\, d\xi\right)^\frac12\\
		&\lesssim \left(\left|\Omega(t-s)\right|^2+\left|\Omega(t-s)\right|\right)
		\left\|e^{(s-\tau)\Delta}\operatorname{div}(u(\tau)\otimes u(\tau))\right\|_{\dot H^m}\\
		&\leq |\Omega|(t-s)\left(1+|\Omega|(t-s)\right)
		(s-\tau)^{-\frac m2-\frac{1-s}2-\frac32(\frac 2q-\frac s3-\frac12)}
		\left\|u(\tau)\right\|_{\dot W^{s, q}}^2.
	\end{align*}
Inserting the estimate into \eqref{eq:71} leads to
	\begin{align*}
		&\left\|\int_0^\frac s2(e^{-(t-s)\Omega\mathbb PJ\mathbb P}-I)e^{(s-\tau)\Delta}e^{-(s-\tau)\Omega\mathbb PJ\mathbb P}\mathbb P\operatorname{div}(u(\tau)\otimes u(\tau))\, d\tau\right\|_{\dot H^m}\\
		&\lesssim |\Omega|(t-s)\left(1+|\Omega|(t-s)\right)
		\varepsilon^{-\frac{m}2}M^{1-\frac2\theta-\frac{1-s}2-\frac32(\frac2q-\frac s3-\frac12)}
		\left\|u\right\|_{L^\theta([0, \infty); \dot W^{s, q})}^2
		\to 0\quad\text{as $t-s\to0$.}
	\end{align*}
	
For \eqref{eq:55}, we have
	\begin{align*}
		&\left\|\int_s^te^{(t-\tau)\Delta}e^{-(t-\tau)\Omega\mathbb PJ\mathbb P}\mathbb P\operatorname{div}(u(\tau)\otimes u(\tau))\, d\tau\right\|_{\dot H^m}\\
		&\lesssim \int_s^t\left\|e^{(t-\tau)\Delta}\operatorname{div}(u(\tau)\otimes u(\tau))\right\|_{\dot H^m}\, d\tau\\
		&\lesssim \int_s^t(t-\tau)^{-\frac{1-s}2-\frac32(\frac 2q-\frac s3-\frac12)}\left\|u(\tau)\right\|_{\dot W^{s+m, q}}\left\|u(\tau)\right\|_{\dot W^{s, q}}\, d\tau\\
		&\lesssim (t-s)^{1-\frac 2\theta-\frac{1-s}2-\frac32(\frac2q-\frac s3-\frac12)}
		\left\|u\right\|_{L^\theta([\varepsilon, \infty); \dot W^{s+m, q})}
		\left\|u\right\|_{L^\theta([0, \infty); \dot W^{s, q})}
		\to 0
		\quad\text{as $t-s\to 0$.}
	\end{align*}
	
For \eqref{eq:57}, similarly to the estimate of \eqref{eq:53}, we have
	\begin{align*}
		&\left\|\int^s_\frac s2(e^{(t-s)\Delta}-I)e^{(s-\tau)\Delta}e^{-(t-\tau)\Omega\mathbb PJ\mathbb P}\mathbb P\operatorname{div}(u(\tau)\otimes u(\tau))\, d\tau\right\|_{\dot H^m}\\
		&\lesssim \int^s_\frac s2\left\|(e^{(t-s)\Delta}-I)e^{(s-\tau)\Delta}\operatorname{div}(u(\tau)\otimes u(\tau))\right\|_{\dot H^m}\, d\tau\\
		&\lesssim (t-s)\int^s_\frac s2(s-\tau)^{-\frac{1-s}2-\frac32(\frac 2q-\frac s3-\frac 12)}\left\|u(\tau)\right\|_{\dot W^{s+m+2, q}}\left\| u(\tau)\right\|_{\dot W^{s, q}}\, d\tau\\
		&\lesssim (t-s)s^{1-\frac 2\theta-\frac{1-s}2-\frac32(\frac 2q-\frac s3-\frac 12)}
		\left\|u\right\|_{L^\theta([\varepsilon, \infty); \dot W^{s+m+2, q})}
		\left\|u\right\|_{L^\theta([0, \infty); \dot W^{s, q})}
		\to 0\quad\text{as $t-s\to0$.}
	\end{align*}

Finally for \eqref{eq:58}, we similarly obtain
	\begin{align*}
		&\left\|\int^s_\frac s2(e^{-(t-s)\Omega\mathbb PJ\mathbb P}-I)e^{(s-\tau)\Delta}e^{-(s-\tau)\Omega\mathbb PJ\mathbb P}\mathbb P\operatorname{div}(u(\tau)\otimes u(\tau))\, d\tau\right\|_{\dot H^m}\\
		&\lesssim |\Omega|(t-s)(1+|\Omega|(t-s))s^{1-\frac 2\theta-\frac{1-s}2-\frac32(\frac 2q-\frac s3-\frac 12)}
		\left\|u\right\|_{L^\theta([\varepsilon, \infty); \dot W^{s+m, q})}
		\left\|u\right\|_{L^\theta([0, \infty); \dot W^{s, q})}\\
		&\to 0\quad\text{as $t-s\to0$.}
	\end{align*}
Combining the above estimates is sufficient to obtain the desired continuity of
	$u(t)$
in
	$\dot H^m(\mathbb R^3)^3.$
\qed

\begin{lemma}
	The solution 
		$u$
	satisfies
		$u\in C([0, \infty); L^2(\mathbb R^3)^3\cap \dot H^s(\mathbb R^3)^3).$
\end{lemma}
\proof
Thanks to Lemma~\ref{lem:6}, it is enough to show the continuity of 
	$u$
at
	$t=0$
in 
	$L^2(\mathbb R^3)^3\cap \dot H^s(\mathbb R^3)^3.$
Let
	$\sigma$
be a real number in 
	$[0, s].$
We have
	\begin{align*}
		\left\|u(t)-u_0\right\|_{\dot H^\sigma}
		\leq \left\|T_\Omega(t)u_0-u_0\right\|_{\dot H^\sigma}
		+\int_0^t\left\|T_\Omega(t-\tau)\mathbb P\operatorname{div}(u(\tau)\otimes u(\tau))\right\|_{\dot H^\sigma}\, d\tau, 
	\end{align*}
and the first term on the right-hand side goes to zero as
	$t\to0, $
since we have
	$u_0\in \dot H^\sigma(\mathbb R^3)^3.$
For the second term, noting the condition
	$\sigma\leq s, $
we see that
	\begin{align*}
		&\int_0^t\left\|T_\Omega(t-\tau)\mathbb P\operatorname{div}(u(\tau)\otimes u(\tau))\right\|_{\dot H^\sigma}\, d\tau\\
		&\lesssim \int_0^t(t-\tau)^{-\frac{1+\sigma-s}2-\frac32(\frac 2q-\frac s3-\frac 12)}\left\|u(\tau)\right\|_{\dot W^{s, q}}^2\, d\tau\\
		&\lesssim  t^{1-\frac2\theta-\frac{1+\sigma-s}2-\frac32(\frac 2q-\frac s3-\frac 12)}
		\left\|u\right\|_{L^\theta([0, \infty); \dot W^{s, q})}^2
		\to 0
		\quad\text{as $t\to0.$}
	\end{align*}
This completes the proof.
\qed

\begin{proposition}
	For any
		$m\geq 0$, 
	the solution
		$u$
	satisfies
		$\partial_tu+\mathbb P(u\cdot \nabla)u -\Delta u+\Omega \mathbb P J \mathbb P u=0$
	in
		$\dot H^m(\mathbb R^3)^3$
	for
		$t>0.$
\end{proposition}
\proof
	In the following argument, we assume that
		$h>0$
	and
		$t>0$
	hold.
	We write
		\begin{align}
			&\frac{u(t+h)-u(t)}{h}
			=\frac{T_\Omega(t+h)u_0-T_\Omega(t)u_0}{h}\notag\\
			&-\frac1{h}\left(\int_0^{t+h}T_\Omega(t+h-\tau)\mathbb P\operatorname{div}(u(\tau)\otimes u(\tau))\, d\tau
			-\int_0^{t}T_\Omega(t+h-\tau)\mathbb P\operatorname{div}(u(\tau)\otimes u(\tau))\, d\tau\right)\quad
			\label{eq:94}\\
			&-\frac1{h}\left(\int_0^{t}T_\Omega(t+h-\tau)\mathbb P\operatorname{div}(u(\tau)\otimes u(\tau))\, d\tau
			-\int_0^{t}T_\Omega(t-\tau)\mathbb P\operatorname{div}(u(\tau)\otimes u(\tau))\, d\tau\right)
			\label{eq:95}
		\end{align}
	and calculate the limits of each line.
	
	For the linear term, we have
		\begin{align}
			&\left\|\frac{T_\Omega(t+h)u_0-T_\Omega(t)u_0}{h}
			-\Delta T_\Omega(t)u_0+\Omega\mathbb P J \mathbb PT_\Omega(t)u_0\right\|_{\dot H^m}\notag\\
			&\leq \left\|\frac{1}{h}\left(e^{h\Delta}e^{-h\Omega\mathbb P J \mathbb P}T_\Omega(t)u_0
			-e^{h\Delta}T_\Omega(t)u_0\right)
			+e^{h\Delta}\Omega\mathbb P J\mathbb P T_\Omega(t)u_0\right\|_{\dot H^m}
			\label{eq:97}\\
			&+ \left\|-e^{h\Delta}\Omega\mathbb P J\mathbb P T_\Omega(t)u_0
			+\Omega\mathbb P J\mathbb P T_\Omega(t)u_0\right\|_{\dot H^m}
			\label{eq:98}\\
			&+\left\|\frac{1}{h}\left(e^{h\Delta}T_\Omega(t)u_0
			-T_\Omega(t)u_0\right)
			-\Delta T_\Omega(t)u_0\right\|_{\dot H^m}.
			\label{eq:99}
		\end{align}
	For \eqref{eq:97}, the mean value theorem yields
		\begin{align*}
			&\left\|\frac{1}{h}\left(e^{-h\Omega\mathbb P J \mathbb P}T_\Omega(t)u_0
			-T_\Omega(t)u_0\right)
			+\Omega\mathbb P J\mathbb P T_\Omega(t)u_0\right\|_{\dot H^m}\\
			&=\left(\int_{\mathbb R^3}\left|\int_0^1
			\left(e^{-\theta h\Omega P(\xi)JP(\xi)}-I\right)\, d\theta
			\mathcal F\left((-\Delta)^\frac m2\Omega\mathbb P J \mathbb P T_\Omega(t)u_0\right)(\xi)\right|^2\right)^\frac 12.
		\end{align*}
	Since
		$T_\Omega(t)u_0\in \dot H^m(\mathbb R^3)^3, $
	the dominated convergence theorem ensures that the right-hand side tends to zero as
		$h\to0.$
	We see that \eqref{eq:98} tends to zero for the same reason. For \eqref{eq:99}, we similarly have
		\begin{align*}
			&\left\|\frac{1}{h}\left(e^{h\Delta}T_\Omega(t)u_0
			-T_\Omega(t)u_0\right)
			-\Delta T_\Omega(t)u_0\right\|_{\dot H^m}\\
			&=\left(\int_{\mathbb R^3}\left|\int_0^1\left(e^{-\theta h|\xi|^2}-1\right)\, d\theta
			\mathcal F\left((-\Delta)^{\frac{m+2}2}T_\Omega(t)u_0\right)(\xi)\right|^2\, d\xi\right)^\frac12
		\end{align*}
	and this also tends to zero as
		$h\to 0$
	by the dominated convergence theorem.
	
	We next go on to the estimate of \eqref{eq:94}. We have
		\begin{align}
			&\left\|\frac1h\int_t^{t+h} T_\Omega(t+h-\tau)\mathbb P\operatorname{div}(u(\tau)\otimes u(\tau))\, d\tau 
			-\mathbb P\operatorname{div}(u(t)\otimes u(t))\right\|_{\dot H^m}\notag\\
			&\leq \left\|\frac1h\int_t^{t+h} T_\Omega(t+h-\tau)\mathbb P\operatorname{div}(u(\tau)\otimes u(\tau)-u(t)\otimes u(t))\, d\tau \right\|_{\dot H^m}\notag\\
			&+\left\|\frac1h\int_t^{t+h} \left(T_\Omega(t+h-\tau)-I\right)\mathbb P\operatorname{div}(u(t)\otimes u(t))\, d\tau \right\|_{\dot H^m}\notag\\
			&\lesssim \sup_{t<\tau<t+h}
			\left\|u(\tau)\otimes u(\tau)
			-u(t)\otimes u(t)\right\|_{\dot H^{m+1}}
			\label{eq:107}\\
			&+\sup_{t<\tau<t+h}
			\left\|\left(T_\Omega(t+h-\tau)-I\right)(u(t)\otimes u(t))\right\|_{\dot H^{m+1}}.
			\label{eq:108}
		\end{align}
	Lemma~\ref{lem:6} gives the continuity of
		$u(\tau)\otimes u(\tau)$
	for
		$\tau>0$
	in
		$H^{m+1}(\mathbb R^3), $
	and	we see that both \eqref{eq:107} and \eqref{eq:108} go to zero by letting
		$h\to0.$
	
	We are now left to calculate the limit of \eqref{eq:95}. For that purpose, we further decompose the concerning term as
		\begin{align}
			&\frac1{h}\left(\int_0^{t}T_\Omega(t+h-\tau)\mathbb P\operatorname{div}(u(\tau)\otimes u(\tau))\, d\tau
			-\int_0^{t}T_\Omega(t-\tau)\mathbb P\operatorname{div}(u(\tau)\otimes u(\tau))\, d\tau\right)\notag\\
			&=\frac1h\int_0^\frac t2 (e^{h\Delta}-I)e^{-h\Omega\mathbb P J \mathbb P}T_\Omega(t-\tau)\mathbb P\operatorname{div}(u(\tau)\otimes u(\tau))\, d\tau
			\label{eq:110}\\
			&+\frac1h\int_0^\frac t2 (e^{-h\Omega\mathbb P J \mathbb P}-I)T_\Omega(t-\tau)\mathbb P\operatorname{div}(u(\tau)\otimes u(\tau))\, d\tau
			\label{eq:111}\\
			&+\frac1h\int^t_\frac t2 (e^{h\Delta}-I)e^{-h\Omega\mathbb P J \mathbb P}T_\Omega(t-\tau)\mathbb P\operatorname{div}(u(\tau)\otimes u(\tau))\, d\tau
			\label{eq:112}\\
			&+\frac1h\int^t_\frac t2 (e^{-h\Omega\mathbb P J \mathbb P}-I)T_\Omega(t-\tau)\mathbb P\operatorname{div}(u(\tau)\otimes u(\tau))\, d\tau
			\label{eq:113}
		\end{align}
	and derive the limits of each line.
	
	For \eqref{eq:110}, we have
		\begin{align}
			&\left\|\frac1h\int_0^\frac t2 (e^{h\Delta}-I)e^{-h\Omega\mathbb P J \mathbb P}T_\Omega(t-\tau)\mathbb P\operatorname{div}(u(\tau)\otimes u(\tau))\, d\tau
			-\int_0^\frac t2\Delta T_\Omega(t-\tau)\mathbb P\operatorname{div}(u(\tau)\otimes u(\tau))\, d\tau\right\|_{\dot H^m}\notag\\
			&\leq \left\|\int_0^\frac t2 \frac{e^{h\Delta}-I}{h}(e^{-h\Omega\mathbb P J \mathbb P}-I)T_\Omega(t-\tau)\mathbb P\operatorname{div}(u(\tau)\otimes u(\tau))\, d\tau\right\|_{\dot H^m}
			\label{eq:115}\\
			&+\left\|\int_0^\frac t2 \left(\frac{e^{h\Delta}-I}{h}-\Delta\right)T_\Omega(t-\tau)\mathbb P\operatorname{div}(u(\tau)\otimes u(\tau))\, d\tau\right\|_{\dot H^m}.
			\label{eq:116}
		\end{align}
	For \eqref{eq:115}, we see that
		\begin{align*}
			&\left\|\int_0^\frac t2 \frac{e^{h\Delta}-I}{h}(e^{-h\Omega\mathbb P J \mathbb P}-I)T_\Omega(t-\tau)\mathbb P\operatorname{div}(u(\tau)\otimes u(\tau))\, d\tau\right\|_{\dot H^m}\\
			&\lesssim h|\Omega| (1+h |\Omega| )\int_0^\frac t2\left\|e^{(t-\tau)\Delta}\operatorname{div}(u(\tau)\otimes u(\tau))\right\|_{\dot H^{m+2}}\, d\tau\\
			&\lesssim h|\Omega| (1+h |\Omega| )\int_0^\frac t2
			(t-\tau)^{-\frac{m+2}2-\frac{1-s}2-\frac32(\frac2q-\frac s3-\frac 12)}
			\left\|u(\tau))\right\|_{\dot W^{s, q}}^2\, d\tau\\
			&\lesssim h|\Omega| (1+h |\Omega| )t^{-\frac {m+2}2+1-\frac 2\theta-\frac{1-s}2-\frac32(\frac2q-\frac s3-\frac 12)}\left\|u\right\|_{L^\theta([0, \infty); \dot W^{s, q})}^2
			\to 0\quad\text{as $h\to0.$}
		\end{align*}
	For \eqref{eq:116}, we have
		\begin{align*}
			&\left\|\int_0^\frac t2 \left(\frac{e^{h\Delta}-I}{h}-\Delta\right)T_\Omega(t-\tau)\mathbb P\operatorname{div}(u(\tau)\otimes u(\tau))\, d\tau\right\|_{\dot H^m}\\
			&\lesssim \int_0^\frac t2 \left\|\left(\frac{e^{h\Delta}-I}{h}-\Delta\right)e^{(t-\tau)\Delta}\operatorname{div}(u(\tau)\otimes u(\tau))\right\|_{\dot H^{m}}\, d\tau.
		\end{align*}
	For each
		$\tau\in(0, \frac t2)$, 
	by the mean value theorem, we see that
		\begin{align*}
			&\left\|\left(\frac{e^{h\Delta}-I}{h}-\Delta\right)e^{(t-\tau)\Delta}\operatorname{div}(u(\tau)\otimes u(\tau))\right\|_{\dot H^{m}}\\
			&=\left(\int_{\mathbb R^3}\left|\int_0^1(e^{-\theta h|\xi|^2}-1)\, d\theta\right|^2\left|\mathcal F\left((-\Delta)^{\frac{m+2}2}e^{(t-\tau)\Delta}\operatorname{div}(u(\tau)\otimes u(\tau))\right)(\xi)\right|^2\, d\xi\right)^\frac12, 
		\end{align*}
	and by the dominated convergence theorem this tends to zero as
		$h\to0.$
	We also have
		\begin{align*}
			&\left\|\left(\frac{e^{h\Delta}-I}{h}-\Delta\right)e^{(t-\tau)\Delta}\operatorname{div}(u(\tau)\otimes u(\tau))\right\|_{\dot H^{m}}\\
			&\lesssim \left\|e^{(t-\tau)\Delta}\operatorname{div}(u(\tau)\otimes u(\tau))\right\|_{\dot H^{m+2}}\\
			&\lesssim (t-\tau)^{-\frac{m+2}2-\frac{1-s}2-\frac32(\frac2q-\frac s3-\frac12)}\left\|u(\tau)\right\|_{\dot W^{s, q}}^2, 
		\end{align*}
	where the bound is integrable on
		$(0, \frac t2).$
	Therefore, applying the dominated convergence theorem again is enough to see that the term \eqref{eq:116} goes to zero as
		$h\to0.$
		
	For \eqref{eq:111}, we write
		\begin{align*}
			&\left\|\int_0^\frac t2 \frac{e^{-h\Omega\mathbb P J \mathbb P}-I}{h}T_\Omega(t-\tau)\mathbb P\operatorname{div}(u(\tau)\otimes u(\tau))\, d\tau
			+\int_0^\frac t2 \Omega \mathbb P J \mathbb PT_\Omega(t-\tau)\mathbb P\operatorname{div}(u(\tau)\otimes u(\tau))\, d\tau\right\|_{\dot H^m}\\
			&\lesssim \int_0^\frac t2 \left\|\left(\frac{e^{-h\Omega\mathbb P J \mathbb P}-I}{h}+\Omega \mathbb P J \mathbb P\right)e^{(t-\tau)\Delta}\operatorname{div}(u(\tau)\otimes u(\tau))\right\|_{\dot H^m}\, d\tau.
		\end{align*}
	Then for
		$\tau\in (0, \frac t2)$, 
	we similarly obtain
		\begin{align*}
			&\left\|\left(\frac{e^{-h\Omega\mathbb P J \mathbb P}-I}{h}+\Omega \mathbb P J \mathbb P\right)e^{(t-\tau)\Delta}\operatorname{div}(u(\tau)\otimes u(\tau))\right\|_{\dot H^m}\\
			&=\left(\int_{\mathbb R^3}\left|\int_0^1\left(e^{-\theta h\Omega P(\xi) J P(\xi)}-I\right)\, d\theta
			\mathcal F\left(\Omega\mathbb P J \mathbb Pe^{(t-\tau)\Delta}\operatorname{div}(u(\tau)\otimes u(\tau))\right)(\xi)\right|^2\, d\xi\right)^\frac 12\\
			&\to0\quad\text{as $h\to0$}
		\end{align*}
	and
		\begin{align*}
			&\left\|\left(\frac{e^{-h\Omega\mathbb P J \mathbb P}-I}{h}+\Omega \mathbb P J \mathbb P\right)e^{(t-\tau)\Delta}\operatorname{div}(u(\tau)\otimes u(\tau))\right\|_{\dot H^m}\\
			&\lesssim |\Omega|(t-\tau)^{-\frac m2-\frac{1-s}2-\frac32(\frac2q-\frac s3-\frac12)}
			\left\|u(\tau)\right\|_{\dot W^{s, q}}^2
			\in L^1((0, t/2)).
		\end{align*}
	Thus applying the dominated convergence ensures that \eqref{eq:111} tends to zero as
		$h\to0.$
	
	The estimates of \eqref{eq:112} and \eqref{eq:113} can be done in almost the same way. The only difference is that the application of the smoothing property of the heat kernel
		$e^{(t-\tau)\Delta}$
	is replaced by that of the smoothing effect of the nonlinear flow (see also the argument in \citep[Proposition~4.4]{Yos26}).
	The resulting estimates are therefore
		\begin{align*}
			&\left\|\int^t_\frac t2 \frac{e^{h\Delta}-I}{h}(e^{-h\Omega\mathbb P J \mathbb P}-I)T_\Omega(t-\tau)\mathbb P\operatorname{div}(u(\tau)\otimes u(\tau))\, d\tau
			-\int^t_\frac t2\Delta T_\Omega(t-\tau)\mathbb P\operatorname{div}(u(\tau)\otimes u(\tau))\, d\tau\right\|_{\dot H^m}\\
			&\to0\quad\text{as $h\to0$}
		\end{align*}
	and
		\begin{align*}
			&\left\|\int^t_\frac t2 \frac{e^{-h\Omega\mathbb P J \mathbb P}-I}{h}T_\Omega(t-\tau)\mathbb P\operatorname{div}(u(\tau)\otimes u(\tau))\, d\tau
			+\int^t_\frac t2 \Omega \mathbb P J \mathbb PT_\Omega(t-\tau)\mathbb P\operatorname{div}(u(\tau)\otimes u(\tau))\, d\tau\right\|_{\dot H^m}\\
			&\to0\quad\text{as $h\to0.$}
		\end{align*}
	
	Combining the above estimates, we obtain
		\begin{align*}
			\lim_{h\downarrow0}\frac{u(t+h)-u(t)}{h}
			=-\mathbb P(u(t)\cdot\nabla)u(t)
			+\Delta u(t)
			-\Omega\mathbb P J \mathbb Pu(t)
		\end{align*}
	for
		$t>0$
	in
		$\dot H^m(\mathbb R^3)^3.$
	It is also possible to verify that the left-derivative of
		$u(t)$
	satisfies the same equation
:
		\begin{align*}
			\lim_{h\downarrow0}\frac{u(t)-u(t-h)}{h}
			=-\mathbb P(u(t)\cdot\nabla)u(t)
			+\Delta u(t)
			-\Omega\mathbb P J \mathbb Pu(t).
		\end{align*}
	As a result, it follows that the solution
		$u$
	satisfies
		$u\in C^1((0, \infty); \dot H^m(\mathbb R^3)^3)$
	and gives a classical solution of the Cauchy problem~\eqref{eq:3} in
		$\dot H^m(\mathbb R^3)^3.$
\qed

\begin{proposition}\label{prop:3.5}
	Let
		$0\leq s\leq t<\infty$
	be given. Then the solution
		$u$
	satisfies the energy equality
		\begin{align*}
			\left\|u(t)\right\|_{L^2}^2
			+2\int_s^t\left\|\nabla u(\tau)\right\|_{L^2}^2\, d\tau
			=\left\|u(s)\right\|_{L^2}^2.
		\end{align*}
\end{proposition}
\proof
	Once we take the 
		$L^2$-inner
	product of the first equation of \eqref{eq:3} and
		$u, $
	the divergence-free condition of
		$u$
	and the skew-symmetry of the operator
		$\mathbb PJ\mathbb P$
	yield
		\begin{align*}
			\frac12\frac{d}{d\tau}\left\|u(\tau)\right\|_{L^2}^2
			+\left\|\nabla u(\tau)\right\|_{L^2}^2
			=0
		\end{align*}
	for
		$\tau>0.$
	Then integrating the equality over 
		$(s, t)$
	gives the conclusion. We note that the calculation is rigorous, since we have
		\begin{align*}
			u\in C([0, \infty); H^s(\mathbb R^3)^3)
			\cap C^1((0, \infty); \dot H^m(\mathbb R^3)^3)
		\end{align*}
	for any
		$m\geq0.$
\qed

\section{Proof of Theorem~\ref{thm:3}: temporal decays}\label{sect:4}
	In this section, we prove the temporal decay estimates of the solution
		$u$.
	We derive the
		$L^q$-type, 
		$L^2$-type, and
		$L^\infty$-type estimates
	successively. Here, we additionally impose the hypothesis that the initial data
		$u_0$
	is an element of
		$L^1(\mathbb R^3)^3$
	, and sometimes we also impose
		$(1+|x|)u_0\in L^1(\mathbb R^3)^3$
	to obtain the additional decay
		$t^{-\frac12}.$
	For the exponents
		$m\geq 0, $
		$p\in[2, \infty], $
	and
		$j\in\{0, 1\}, $
	we define the functions
		$M^{m, p}_j(t)$
	and
		$U^{m, p}_j(t)$
	for
		$t\in(0, \infty)$
	by
		\begin{align*}
			M^{m, p}_j(t)
			\coloneqq t^{-\frac j2-\frac m2-\frac32(1-\frac 1p)}(1+|\Omega|t)^{-1+\frac2p}
		\end{align*}
	and
		\begin{align*}
			U^{m, p}_j(t)
			\coloneqq \sup_{0<\tau\leq t}
			|\Omega|^{1+\frac j2}\tau^{\frac j2+\frac m2+\frac32(1-\frac 1p)}
			(1+|\Omega|\tau)^{1-\frac2p}
			\left\|u(\tau)\right\|_{\dot W^{m, p}}.
		\end{align*}
	Then the proof of Theorem~\ref{thm:3} is rephrased as showing the boundedness the function
		$U^{m, p}_{j}$
	on
		$(0, \infty).$
	Note that
	$U^{m, p}_j(t)<\infty$
holds for all
	$t>0$
by Lemma~\ref{lem:6} and the interpolation inequality.

\subsection{$L^q$-type estimates}
We first consider the case
	$p=q, $
where
	$q$
is the exponent which appears in the statement of Theorem~\ref{thm:1}. 
\begin{proposition}\label{prop:10}
	Let
		$u$
	be the global solution obtained in Theorem~\ref{thm:1} and assume that the initial data
		$u_0$
	satisfies
		$u_0\in L^1(\mathbb R^3)^3\cap \dot H^s(\mathbb R^3)^3.$
	Then there are positive constants
		$C_{**}(\leq C_*)$
	and
		$C$
	depending only on
		$s, q, $
	and
		$\theta$
	such that if
		$|\Omega|^{-\frac12(s-\frac12)}\left\|u_0\right\|_{\dot H^s}
		\leq C_{**}, $
	the estimate
		\begin{align*}
			U^{0, q}_0(t)
			\leq C\left(|\Omega|\left\|u_0\right\|_{L^1}
			+|\Omega|\left\|u_0\right\|_{L^2}^4\right)
		\end{align*}
	holds for all
		$t>0.$
\end{proposition}
\proof
By the triangle inequality we have
	\begin{align*}
		\left\|u(t)\right\|_{L^q}
		&\leq \left\|T_\Omega(t)u_0\right\|_{L^q}
		+\int_0^\frac t2\left\|T_\Omega(t-\tau)\mathbb P\operatorname{div}(u(\tau)\otimes u(\tau))\right\|_{L^q}\, d\tau\\
		&+\int_\frac t2^t\left\|T_\Omega(t-\tau)\mathbb P\operatorname{div}(u(\tau)\otimes u(\tau))\right\|_{L^q}\, d\tau
	\end{align*}
and we treat the estimates of each term on the right-hand side. 

Applying Proposition~\ref{prop:2.2}, we obtain the linear decay
	\begin{align*}
		\left\|T_\Omega(t)u_0\right\|_{L^q}
		\lesssim t^{-\frac32(1-\frac1q)}(1+|\Omega|t)^{-1+\frac 2q}\left\|u_0\right\|_{L^1}
		=M^{0, q}_0(t)\left\|u_0\right\|_{L^1}.
	\end{align*}
	
For the first half of the Duhamel integral, by Proposition~\ref{prop:2.2} we have
	\begin{align*}
		&\int_0^\frac t2\left\|T_\Omega(t-\tau)\mathbb P\operatorname{div}(u(\tau)\otimes u(\tau))\right\|_{L^q}\, d\tau\\
		&\lesssim \int_0^\frac t2(1+|\Omega|(t-\tau))^{-1+\frac2q}(t-\tau)^{-\frac32(1-\frac 1q)-\frac 12}\left\|u(\tau)\otimes u(\tau)\right\|_{L^1}\, d\tau\\
		&\lesssim M^{0, q}_0(t)t^{-\frac 12}\int_0^\frac t2\left\|u(\tau)\right\|_{L^2}\, d\tau
		\lesssim M^{0, q}_0(t)t^{\frac12}\left\|u_0\right\|_{L^2}^2.
	\end{align*}
Since the assumption~\eqref{eq:5} ensures that
	$q\in(2, 3), $
we may also derive
	\begin{align*}
		&\int_0^\frac t2\left\|T_\Omega(t-\tau)\mathbb P\operatorname{div}(u(\tau)\otimes u(\tau))\right\|_{L^q}\, d\tau\\
		&\lesssim \int_0^\frac t2(1+|\Omega|(t-\tau))^{-1+\frac2q}(t-\tau)^{-\frac32(1-\frac 2q)-\frac 12}\left\|e^{\frac12(t-\tau)\Delta}(u(\tau)\otimes u(\tau))\right\|_{L^{q'}}\, d\tau\\
		&\lesssim \int_0^\frac t2(1+|\Omega|(t-\tau))^{-1+\frac2q}(t-\tau)^{-\frac32(1-\frac 2q)-\frac 12-\frac32(\frac12+\frac1q-1+\frac1q)}\left\|u(\tau)\right\|_{L^q}\left\|u(\tau)\right\|_{L^2}\, d\tau\\
		&\lesssim (1+|\Omega|t)^{-1+\frac2q}t^{-\frac54}\left\|u_0\right\|_{L^2}
		\int_0^\frac t2\left\|u(\tau)\right\|_{L^q}\, d\tau\\
		&\lesssim (1+|\Omega|t)^{-1+\frac2q}t^{-\frac54}\left\|u_0\right\|_{L^2}|\Omega|^{-1}U^{0, q}_0(t)
		\int_0^\frac t2\tau^{-\frac32(1-\frac1q)}\, d\tau\\
		&\lesssim M^{0, q}_0(t)t^{-\frac14}\left\|u_0\right\|_{L^2}|\Omega|^{-1}U^{0, q}_0(t), 
	\end{align*}
since we have
	$\frac12+\frac1q-1+\frac1q>0$
and
	$\frac32(1-\frac1q)<1.$
By interpolation, we consequently have
	\begin{align*}
		&\int_0^\frac t2\left\|T_\Omega(t-\tau)\mathbb P\operatorname{div}(u(\tau)\otimes u(\tau))\right\|_{L^q}\, d\tau\\
		&\lesssim M^{0, q}_0(t)\left\|u_0\right\|_{L^2}\left\|u_0\right\|_{L^2}^\frac13|\Omega|^{-\frac23}U^{0, q}_0(t)^\frac23
		=|\Omega|^{-1}M^{0, q}_0(t)\left(|\Omega|^{\frac14}\left\|u_0\right\|_{L^2}\right)^\frac43U^{0, q}_0(t)^\frac23.
	\end{align*}

The treatment of the second half of the Duhamel integral is the same as in \citep[Lemma~5]{AKL22} and \citep[Lemma~4.3]{ET23}. 
Indeed, since the conditions
	$\frac2q-\frac s3\geq 1-\frac1q$
and
	$\frac12-\frac32(\frac1q-\frac s3)-(1-\frac2q)
	<\frac1\theta
	< \frac12-\frac32(\frac1q-\frac s3)$
hold, 
we have
	\begin{align*}
		&\int_\frac t2^t\left\|T_\Omega(t-\tau)\mathbb P\operatorname{div}(u(\tau)\otimes u(\tau))\right\|_{L^q}\, d\tau\\
		&\lesssim \int_\frac t2^t (1+|\Omega|(t-\tau))^{-1+\frac 2q}(t-\tau)^{-\frac12-\frac32(1-\frac2q)}
		\left\|e^{\frac12(t-\tau)\Delta}(u(\tau)\otimes u(\tau))\right\|_{L^{q'}}\, d\tau\\
		&\lesssim \int_\frac t2^t (1+|\Omega|(t-\tau))^{-1+\frac 2q}(t-\tau)^{-\frac12-\frac32(1-\frac2q+\frac2q-\frac s3-1+\frac1q)}
		\left\|u(\tau)\right\|_{L^q} \left\|u(\tau)\right\|_{\dot W^{s, q}}\, d\tau\\
		&\lesssim |\Omega|^{-1}M^{0, q}_0(t)U^{0, q}_0(t)
		\int_\frac t2^t (1+|\Omega|(t-\tau))^{-1+\frac 2q}(t-\tau)^{-\frac12-\frac32(\frac1q-\frac s3)}
		 \left\|u(\tau)\right\|_{\dot W^{s, q}}\, d\tau\\
		 &\leq |\Omega|^{-1}M^{0, q}_0(t)U^{0, q}_0(t)
		 \left\|u\right\|_{L^\theta([0, \infty); \dot W^{q, s})}
		 \left(\int_\frac t2^t\left( (1+|\Omega|(t-\tau))^{-1+\frac 2q}(t-\tau)^{-\frac12-\frac32(\frac1q-\frac s3)}\right)^{\frac\theta{\theta-1}}\, d\tau\right)^{1-\frac1\theta}\\
		 &\lesssim |\Omega|^{-1}M^{0, q}_0(t)U^{0, q}_0(t)
		 |\Omega|^{-\frac12(s-\frac12)}\left\|u_0\right\|_{\dot H^s}.
	\end{align*}
Putting the above estimates altogether, we obtain
	\begin{align*}
		\left\|u(t)\right\|_{L^q}
		\lesssim |\Omega|^{-1}M^{0, q}_0(t)\left(|\Omega|\left\|u_0\right\|_{L^1}
		+\left(|\Omega|^{\frac14}\left\|u_0\right\|_{L^2}\right)^\frac43U^{0, q}_0(t)^\frac23
		+U^{0, q}_0(t)
		 |\Omega|^{-\frac12(s-\frac12)}\left\|u_0\right\|_{\dot H^s}\right),
	\end{align*}
and thus by taking
	$|\Omega|^{-\frac12(s-\frac12)}\left\|u_0\right\|_{\dot H^s}$
sufficiently small, we have
	\begin{align*}
		U^{0, q}_0(t)
		\lesssim 
		|\Omega|\left\|u_0\right\|_{L^1}
		+\left(|\Omega|^{\frac14}\left\|u_0\right\|_{L^2}\right)^\frac43U^{0, q}_0(t)^\frac23.
	\end{align*}
This estimate in turn yields
	\begin{align*}
		U^{0, q}_0(t)
		\lesssim |\Omega|\left\|u_0\right\|_{L^1}
		+|\Omega|\left\|u_0\right\|_{L^2}^4.
	\end{align*}
This completes the proof.
\qed

\begin{proposition}\label{prop:11}
	Let
		$u$
	be the global solution obtained in Theorem~\ref{thm:1} and assume that the initial data
		$u_0$
	satisfies
		$u_0\in L^1(\mathbb R^3)^3\cap \dot H^s(\mathbb R^3)^3.$
	Then there are positive constants
		$C_{\star}(\leq C_{**})$
	and
		$C$
	depending only on
		$s, q, $
	and
		$\theta$
	such that if
		$|\Omega|^{-\frac12(s-\frac12)}\left\|u_0\right\|_{\dot H^s}
		\leq C_\star, $
	the estimate
		\begin{align*}
			U^{1, q}_0(t)
			\leq C\left(|\Omega|\left\|u_0\right\|_{L^1}
			+\left(|\Omega|^\frac14\left\|u_0\right\|_{L^2}\right)^\frac43
			U^{0, q}_0(t)^\frac{2}{3}\right)
		\end{align*}
	holds for all
		$t>0.$
\end{proposition}
\proof
We write
	\begin{align*}
		\left\|u(t)\right\|_{\dot W^{1, q}}
		&\leq \left\|T_\Omega(t)u_0\right\|_{\dot W^{1, q}}
		+\int_0^\frac t2\left\|T_\Omega(t-\tau)\mathbb P\operatorname{div}(u(\tau)\otimes u(\tau))\right\|_{\dot W^{1, q}}\, d\tau\\
		&+\int_\frac t2^t\left\|T_\Omega(t-\tau)\mathbb P\operatorname{div}(u(\tau)\otimes u(\tau))\right\|_{\dot W^{1, q}}\, d\tau, 
	\end{align*}
and thanks to the smoothing property of the heat kernel, as in Proposition~\ref{prop:10}, we obtain
	\begin{align*}
		&\left\|T_\Omega(t)u_0\right\|_{\dot W^{1, q}}
		+\int_0^\frac t2\left\|T_\Omega(t-\tau)\mathbb P\operatorname{div}(u(\tau)\otimes u(\tau))\right\|_{\dot W^{1, q}}\, d\tau\\
		&\lesssim |\Omega|^{-1}M^{1, q}_0(t)\left(\left\|u_0\right\|_{L^1}
		+\left(|\Omega|^\frac14\left\|u_0\right\|_{L^2}\right)^\frac43
			U^{0, q}_0(t)^\frac{2}{3}\right).
	\end{align*}
For the remaining term, an estimate similar to one in the proof of Proposition~\ref{prop:10} is enough to have
	\begin{align*}
		&\int_\frac t2^t\left\|T_\Omega(t-\tau)\mathbb P\operatorname{div}(u(\tau)\otimes u(\tau))\right\|_{\dot W^{1, q}}\, d\tau\\
		&\lesssim \int_\frac t2^t (1+|\Omega|(t-\tau))^{-1+\frac 2q}(t-\tau)^{-\frac12-\frac32(1-\frac2q)}
		\left\|e^{\frac12(t-\tau)\Delta}(-\Delta)^\frac12(u(\tau)\otimes u(\tau))\right\|_{L^{q'}}\, d\tau\\
		&\lesssim \int_\frac t2^t (1+|\Omega|(t-\tau))^{-1+\frac 2q}(t-\tau)^{-\frac12-\frac32(1-\frac2q+\frac2q-\frac s3-1+\frac1q)}
		\left\|u(\tau)\right\|_{\dot W^{1, q}} \left\|u(\tau)\right\|_{\dot W^{s, q}}\, d\tau\\
		 &\lesssim |\Omega|^{-1}M^{1, q}_0(t)U^{1, q}_0(t)
		 |\Omega|^{-\frac12(s-\frac12)}\left\|u_0\right\|_{\dot H^s}.
	\end{align*}
Therefore by taking
	$|\Omega|^{-\frac12(s-\frac12)}\left\|u_0\right\|_{\dot H^s}$
sufficiently small, we obtain the desired estimate
	\begin{align*}
		U^{1, q}_0(t)
			\leq C\left(|\Omega|\left\|u_0\right\|_{L^1}
			+\left(|\Omega|^\frac14\left\|u_0\right\|_{L^2}\right)^\frac43
			U^{0, q}_0(t)^\frac{2}{3}\right).
	\end{align*}
\qed

\begin{proposition}
	Let
		$u$
	be the global solution obtained in Theorem~\ref{thm:1} and assume that the initial data
		$u_0$
	satisfies
		$u_0\in L^1(\mathbb R^3)^3\cap \dot H^s(\mathbb R^3)^3$
	and
		$|\Omega|^{-\frac12(s-\frac12)}\left\|u_0\right\|_{\dot H^s}
		\leq C_\star.$
	Then for any
		$m\geq 0$
	there exists a constant
		$C>0$
	depending only on
		$m, $
		$s, $
		$q, $
		$\theta, $
		$|\Omega|\left\|u_0\right\|_{L^1}, $
		$|\Omega|^\frac14\left\|u_0\right\|_{L^2}, $
	and
		$|\Omega|^{-\frac12(s-\frac12)}\left\|u_0\right\|_{\dot H^s}, $
	satisfying
		$U^{1+m, q}_0(t)\leq C$
	for all
		$t>0.$
\end{proposition}
\proof
The proof is based on inductive arguments which is already used in Lemma~\ref{lem:5}. We beforehand take a sufficiently small number
	$\alpha\in(0, 1)$
satisfying
	\begin{align*}
		\frac12-\frac\alpha2-\frac32\left(\frac1q-\frac s3\right)-\left(1-\frac2q\right)
		<\frac1\theta
		<\frac12-\frac\alpha2-\frac32\left(\frac1q-\frac s3\right). 
	\end{align*}
Then we show that
	$U^{1+k\alpha, q}_0(t)
	\leq C
	$
holds for all
	$t>0$
and
	$k\in\mathbb N\cup\{0\}$
by induction of 
	$k$.
The base case
	$k=0$
follows by Proposition~\ref{prop:11}. We thus consider the case of
	$k\geq 1.$

As usual, we write
	\begin{align*}
		\left\|u(t)\right\|_{\dot W^{1+k\alpha, q}}
		&\leq \left\|T_\Omega(t)u_0\right\|_{\dot W^{1+k\alpha, q}}
		+\int_0^\frac t2\left\|T_\Omega(t-\tau)\mathbb P\operatorname{div}(u(\tau)\otimes u(\tau))\right\|_{\dot W^{1+k\alpha, q}}\, d\tau\\
		&+\int_\frac t2^t\left\|T_\Omega(t-\tau)\mathbb P\operatorname{div}(u(\tau)\otimes u(\tau))\right\|_{\dot W^{1+k\alpha, q}}\, d\tau, 
	\end{align*}
and by the smoothing property of the heat kernel, we obtain
	\begin{align*}
		&\left\|T_\Omega(t)u_0\right\|_{\dot W^{1+k\alpha, q}}
		+\int_0^\frac t2\left\|T_\Omega(t-\tau)\mathbb P\operatorname{div}(u(\tau)\otimes u(\tau))\right\|_{\dot W^{1+k\alpha, q}}\, d\tau\\
		&\lesssim |\Omega|^{-1}M^{1+k\alpha, q}_0(t)\left(\left\|u_0\right\|_{L^1}
		+\left(|\Omega|^\frac14\left\|u_0\right\|_{L^2}\right)^\frac43
			U^{0, q}_0(t)^\frac{2}{3}\right).
	\end{align*}

For the second half of the Duhamel integral, we have
	\begin{align*}
		&\int_\frac t2^t\left\|T_\Omega(t-\tau)\mathbb P\operatorname{div}(u(\tau)\otimes u(\tau))\right\|_{\dot W^{1+k\alpha, q}}\, d\tau\\
		&\lesssim \int_\frac t2^t(1+|\Omega|(t-\tau))^{-1+\frac 2q}(t-\tau)^{-\frac12-\frac\alpha2-\frac32(\frac1q-\frac s3)}\left\|u(\tau)\right\|_{\dot W^{1+(k-1)\alpha, q}}\left\|u(\tau)\right\|_{\dot W^{s, q}}\, d\tau\\
		&\lesssim |\Omega|^{-1}M^{1+(k-1)\alpha, q}_0(t)U^{1+(k-1)\alpha, q}_0(t)
		\int_\frac t2^t(1+|\Omega|(t-\tau))^{-1+\frac 2q}(t-\tau)^{-\frac12-\frac\alpha2-\frac32(\frac1q-\frac s3)}\left\|u(\tau)\right\|_{\dot W^{s, q}}\, d\tau.
	\end{align*}
For the integral factor, we obtain
	\begin{align*}
		&\int_\frac t2^t(1+|\Omega|(t-\tau))^{-1+\frac 2q}(t-\tau)^{-\frac12-\frac\alpha2-\frac32(\frac1q-\frac s3)}\left\|u(\tau)\right\|_{\dot W^{s, q}}\, d\tau\\
		&\leq \left\|u\right\|_{L^\theta([0, \infty); \dot W^{s, q})}\left(\int_\frac t2^t\left((1+|\Omega|(t-\tau))^{-1+\frac 2q}(t-\tau)^{-\frac12-\frac\alpha2-\frac32(\frac1q-\frac s3)}\right)^\frac\theta{\theta-1}\, d\tau\right)^{1-\frac1\theta}\\
		&\lesssim |\Omega|^{-\frac1\theta+\frac34(1-\frac2q)}\left\|u_0\right\|_{\dot H^s}|\Omega|^{\frac12+\frac\alpha2 +\frac32(\frac1q-\frac s3)-1+\frac1\theta}
		=|\Omega|^\frac\alpha2|\Omega|^{-\frac12(s-\frac12)}\left\|u_0\right\|_{\dot H^s}
	\end{align*}
and
	\begin{align*}
		&\int_\frac t2^t(1+|\Omega|(t-\tau))^{-1+\frac 2q}(t-\tau)^{-\frac12-\frac\alpha2-\frac32(\frac1q-\frac s3)}\left\|u(\tau)\right\|_{\dot W^{s, q}}\, d\tau\\
		&\lesssim |\Omega|^{-1}M^{s, q}_0(t)U^{s, q}_0(t)t^{\frac12-\frac\alpha2-\frac32(\frac1q-\frac s3)}
		\leq |\Omega t|^{-1}t^{-\frac\alpha2}U^{s, q}_0(t).
	\end{align*}
Thus by interpolation we have
	\begin{align*}
		&\int_\frac t2^t\left\|T_\Omega(t-\tau)\mathbb P\operatorname{div}(u(\tau)\otimes u(\tau))\right\|_{\dot W^{1+k\alpha, q}}\, d\tau\\
		&\lesssim |\Omega|^{-1}M^{1+(k-1)\alpha, q}_0(t)U^{1+(k-1)\alpha, q}_0(t)
		\min\left\{|\Omega|^\frac\alpha2|\Omega|^{-\frac12(s-\frac12)}\left\|u_0\right\|_{\dot H^s}, 
		|\Omega t|^{-1}t^{-\frac\alpha2}U^{s, q}_0(t)\right\}\\
		&\leq |\Omega|^{-1}M^{1+k\alpha, q}_0(t)U^{1+(k-1)\alpha, q}_0(t)
		\left(|\Omega|^{-\frac12(s-\frac12)}\left\|u_0\right\|_{\dot H^s}\right)^{\frac{2}{2+\alpha}}
		U^{s, q}_0(t)^\frac{\alpha}{2+\alpha}.
	\end{align*}
Since
	$0<s<1, $
interpolating the estimates of Proposition~\ref{prop:10} and Proposition~\ref{prop:11} gives the boundedness of
	$U^{s, q}_0(t)$
for
	$t\in(0, \infty).$
The hypothesis of induction also ensures the boundedness of the term
	$U^{1+(k-1)\alpha, q}_0(t), $
and thus we obtain the desired conclusion.
\qed

\subsection{$L^2$-type estimates}
\begin{proposition}\label{prop:13}
	Let
		$u$
	be the global solution obtained in Theorem~\ref{thm:1} and assume that the initial data
		$u_0$
	satisfies
		$u_0\in L^1(\mathbb R^3)^3\cap \dot H^s(\mathbb R^3)^3$
	and
		$|\Omega|^{-\frac12(s-\frac12)}\left\|u_0\right\|_{\dot H^s}
		\leq C_\star.$
	Then for any
		$m\geq 0$
	there exists a constant
		$C>0$
	depending only on
		$m, $
		$s, $
		$q, $
		$\theta, $
		$|\Omega|\left\|u_0\right\|_{L^1}, $
		$|\Omega|^\frac14\left\|u_0\right\|_{L^2}, $
	and
		$|\Omega|^{-\frac12(s-\frac12)}\left\|u_0\right\|_{\dot H^s}, $
	satisfying
		$U^{m, 2}_0(t)\leq C$
	for all
		$t>0.$
\end{proposition}
\proof
For the linear term, by Proposition~\ref{prop:2.2} we have
	\begin{align*}
		\left\|T_\Omega(t)u_0\right\|_{\dot H^m}
		\lesssim M^{m, 2}_0(t)\left\|u_0\right\|_{L^1}.
	\end{align*}
For the first half of the Duhamel integral, using Proposition~\ref{prop:2.2}, the Cauchy--Schwarz inequality, and the energy equality, we have
	\begin{align*}
		&\int_0^\frac t2\left\|T_\Omega(t-\tau)\mathbb P\operatorname{div}(u(\tau)\otimes u(\tau))\right\|_{\dot H^m}\, d\tau\\
		&\lesssim t^{-\frac m2-\frac34-\frac12}\int_0^\frac t2\left\|u(\tau)\otimes u(\tau)\right\|_{L^1}\, d\tau\\
		&\lesssim M^{m, 2}_0(t)t^{-\frac12}\int_0^\frac t2\left\|u(\tau)\right\|_{L^2}^2\, d\tau
		\lesssim |\Omega|^{-1}M^{m, 2}_0(t)(|\Omega| t)^\frac12(|\Omega|^\frac14\left\|u_0\right\|_{L^2})^2
	\end{align*}
and
	\begin{align}
		&\int_0^\frac t2\left\|T_\Omega(t-\tau)\mathbb P\operatorname{div}(u(\tau)\otimes u(\tau))\right\|_{\dot H^m}\, d\tau\notag\\
		&\lesssim M^{m, 2}_0(t)t^{-\frac12}\int_0^\frac t2\left\|u(\tau)\right\|_{L^2}^2\, d\tau
		\lesssim M^{m, 2}_0(t)t^{-\frac12}\left\|u_0\right\|_{L^2}
		|\Omega|^{-1}U^{0, 2}_0(t)\int_0^\frac t2\tau^{-\frac34}\, d\tau\notag\\
		&\lesssim |\Omega|^{-1}M^{m, 2}_0(t)(|\Omega|t)^{-\frac14}|\Omega|^\frac14\left\|u_0\right\|_{L^2}
		U^{0, 2}_0(t).\label{eq:a}
	\end{align}
Consequently by interpolation we obtain
	\begin{align*}
		&\int_0^\frac t2\left\|T_\Omega(t-\tau)\mathbb P\operatorname{div}(u(\tau)\otimes u(\tau))\right\|_{\dot H^m}\, d\tau
		\lesssim |\Omega|^{-1}M^{m, 2}_0(t)
		\left(|\Omega|^\frac14\left\|u_0\right\|_{L^2}\right)^\frac43U^{0, 2}_0(t)^\frac23.
	\end{align*}

For the second half of the Duhamel integral, we have
	\begin{align*}
		&\int_\frac t2^t\left\|T_\Omega(t-\tau)\mathbb P\operatorname{div}(u(\tau)\otimes u(\tau))\right\|_{\dot H^m}\, d\tau\\
		&\lesssim \int_\frac t2^t(t-\tau)^{-\frac12-\frac32(\frac2q-\frac s3-\frac12)}
		\left\|u(\tau)\right\|_{\dot W^{m, q}} \left\|u(\tau)\right\|_{\dot W^{s, q}}\, d\tau\\
		&\lesssim |\Omega|^{-1}M^{m, q}_0(t)U^{m, q}_0(t)
		\left\|u\right\|_{L^\theta([0, \infty); \dot W^{s, q})}
		\left(\int_\frac t2^t\left((t-\tau)^{-\frac12-\frac32(\frac2q-\frac s3-\frac12)}\right)^{\frac\theta{\theta-1}}
		\, d\tau\right)^{1-\frac1\theta}.
	\end{align*}
Since the assumption
	$\frac1\theta
	<\frac58-\frac3{2q}+\frac s4$
is equivalent to
	$\frac12+\frac32(\frac 2q-\frac s3+\frac12)
	<1-\frac 2\theta$
and we trivially have
	$1-\frac2\theta
	<1-\frac1\theta, $
the integral is finite and we obtain
	\begin{align}
		&\int_\frac t2^t\left\|T_\Omega(t-\tau)\mathbb P\operatorname{div}(u(\tau)\otimes u(\tau))\right\|_{\dot H^m}\, d\tau\notag\\
		&\lesssim |\Omega|^{-1}M^{m, q}_0(t)U^{m, q}_0(t)
		\left\|u\right\|_{L^\theta([0, \infty); \dot W^{s, q})}
		\left(\int_\frac t2^t\left((t-\tau)^{-\frac12-\frac32(\frac2q-\frac s3-\frac12)}\right)^{\frac\theta{\theta-1}}
		\, d\tau\right)^{1-\frac1\theta}\notag\\
		&\lesssim |\Omega|^{-1}U^{m, q}_0(t)
		|\Omega|^{-\frac1\theta+\frac34(1-\frac2q)}\left\|u_0\right\|_{\dot H^s}
		t^{-\frac m2-\frac32(1-\frac1q)+\frac12-\frac1\theta-\frac32(\frac2q-\frac s3-\frac12)}
		(1+|\Omega|t)^{-1+\frac2q}\notag\\
		&=|\Omega|^{-1}
		M^{m, 2}_0(t)
		(|\Omega|t)^{-\frac1\theta+\frac12-\frac{3}{2q}+\frac s2}(1+|\Omega|t)^{-1+\frac2q}
		U^{m, q}_0(t)|\Omega|^{-\frac12(s-\frac12)}\left\|u_0\right\|_{\dot H^s}\label{eq:b}\\
		&\leq|\Omega|^{-1}
		M^{m, 2}_0(t)
		U^{m, q}_0(t)|\Omega|^{-\frac12(s-\frac12)}\left\|u_0\right\|_{\dot H^s}, \notag
	\end{align}
where we used the condition
	$\frac12-\frac32(\frac1q-\frac s3)-(1-\frac2q)
	<\frac1\theta
	<\frac12-\frac32(\frac1q-\frac s3).$

Combining the above estimates, we obtain
	\begin{align*}
		U^{m, 2}_0(t)
		\lesssim |\Omega|\left\|u_0\right\|_{L^1}
		+(|\Omega|^\frac14\left\|u_0\right\|_{L^2})^\frac43U^{0, 2}_0(t)^\frac23
		+U^{m, q}_0(t)
		|\Omega|^{-\frac12(s-\frac12)}\left\|u_0\right\|_{\dot H^s}.
	\end{align*}
If
	$m=0$, 
the estimate yields
	\begin{align*}
		U^{0, 2}_0(t)
		\lesssim |\Omega|\left\|u_0\right\|_{L^1}
		+|\Omega|\left\|u_0\right\|_{L^2}^4
		+U^{0, q}_0(t)
		|\Omega|^{-\frac12(s-\frac12)}\left\|u_0\right\|_{\dot H^s}.
	\end{align*}
Inserting this to the original one, we obtain the desired estimates for general
	$m>0.$
\qed

\begin{proposition}
	Let
		$u$
	be the global solution obtained in Theorem~\ref{thm:1} and assume that the initial data
		$u_0$
	satisfies
		$u_0\in L^1(\mathbb R^3)^3\cap \dot H^s(\mathbb R^3)^3, $
		$|x|u_0\in L^1(\mathbb R^3)^3, $
	and
		$|\Omega|^{-\frac12(s-\frac12)}\left\|u_0\right\|_{\dot H^s}
		\leq C_\star.$
	Then for any
		$m\geq 0$
	there exists a constant
		$C>0$
	depending only on
		$m, $
		$s, $
		$q, $
		$\theta, $
		$|\Omega|\left\|u_0\right\|_{L^1}, $
		$|\Omega|^{\frac32}\left\||x|u_0\right\|_{L^1}, $
		$|\Omega|^\frac14\left\|u_0\right\|_{L^2}, $
	and
		$|\Omega|^{-\frac12(s-\frac12)}\left\|u_0\right\|_{\dot H^s}, $
	satisfying
		$U^{m, 2}_1(t)\leq C$
	for all
		$t>0.$
\end{proposition}
\proof
For the linear term, by Proposition~\ref{prop:2.4}, we have
	\begin{align*}
		\left\|T_\Omega(t)u_0\right\|_{\dot H^m}
		\lesssim M^{m, 2}_1(t)\left\||x|u_0\right\|_{L^1}.
	\end{align*}
For the first half of the Duhamel integral, we have
	\begin{align*}
		&\int_0^\frac t2\left\|T_\Omega(t-\tau)\mathbb P\operatorname{div}(u(\tau)\otimes u(\tau))\right\|_{\dot H^m}\, d\tau\\
		&\lesssim M^{m, 2}_0(t)t^{-\frac12}
		\int_0^\frac t2\left\|u(\tau)\otimes u(\tau)\right\|_{L^1}\, d\tau
		\lesssim M^{m, 2}_1(t)\int_0^\infty\left\|u(\tau)\right\|_{L^2}^2\, d\tau.
	\end{align*}
In order to estimate the integral term, let
	$\lambda>0$
be an arbitrary positive real number and we split the integral at
	$\tau=\lambda$
to obtain
	\begin{align*}
		&\int_0^\infty\left\|u(\tau)\right\|_{L^2}^2\, d\tau
		=\int_0^\lambda\left\|u(\tau)\right\|_{L^2}^2\, d\tau
		+\int_\lambda^\infty\left\|u(\tau)\right\|_{L^2}^2\, d\tau\\
		&\leq\left\|u_0\right\|_{L^2}^2\int_0^\lambda\, d\tau
		+|\Omega|^{-2}\big\|U^{0, 2}_0\big\|_{L^\infty}^2
		\int_\lambda^\infty \tau^{-\frac32}\, d\tau
		\lesssim \left\|u_0\right\|_{L^2}^2\lambda
		+|\Omega|^{-2}\big\|U^{0, 2}_0\big\|_{L^\infty}^2\lambda^{-\frac12}.
	\end{align*}
Here, by setting
	$\left\|u_0\right\|_{L^2}^2\lambda
	=|\Omega|^{-2}\big\|U^{0, 2}_0\big\|_{L^\infty}^2\lambda^{-\frac12}, $
i.e., 
	$\lambda
	=\left\|u_0\right\|_{L^2}^{-\frac43}|\Omega|^{-\frac43}
	\big\|U^{0, 2}_0\big\|_{L^\infty}^\frac43, $
we have
	\begin{align*}
		&\int_0^\infty\left\|u(\tau)\right\|_{L^2}^2\, d\tau
		\lesssim \left\|u_0\right\|_{L^2}^{\frac23}|\Omega|^{-\frac43}
	\big\|U^{0, 2}_0\big\|_{L^\infty}^\frac43, 
	\end{align*}
and thus
	\begin{align*}
		&\int_0^\frac t2\left\|T_\Omega(t-\tau)\mathbb P\operatorname{div}(u(\tau)\otimes u(\tau))\right\|_{\dot H^m}\, d\tau
		\lesssim |\Omega|^{-\frac32}M^{m, 2}_1(t)
		\left(|\Omega|^\frac14\left\|u_0\right\|_{L^2}\right)^{\frac23}
	\big\|U^{0, 2}_0\big\|_{L^\infty}^\frac43.
	\end{align*}
	
For the second half of the Duhamel integral,
as in Proposition~\ref{prop:13}, we have
	\begin{align*}
		&\int_\frac t2^t\left\|T_\Omega(t-\tau)\mathbb P\operatorname{div}(u(\tau)\otimes u(\tau))\right\|_{\dot H^m}\, d\tau\\
		&\lesssim \int_\frac t2^t(t-\tau)^{-\frac12-\frac32(\frac2q-\frac s3-\frac12)}
		\left\|u(\tau)\right\|_{\dot W^{m, q}} \left\|u(\tau)\right\|_{\dot W^{s, q}}\, d\tau\\
		&\lesssim |\Omega|^{-1}M^{m, q}_0(t)U^{m, q}_0(t)
		|\Omega|^{-1}M^{s, q}_0(t)U^{s, q}_0(t)t^{\frac12-\frac32(\frac2q-\frac s3-\frac12)}\\
		&\leq |\Omega|^{-2}U^{m, q}_0(t)U^{s, q}_0(t)
		t^{-\frac m2-\frac74}
		=|\Omega|^{-\frac32}M^{m, 2}_1(t)(|\Omega|t)^{-\frac12}U^{m, q}_0(t)U^{s, q}_0(t).
	\end{align*}
Interpolating this estimate with the one which is previously obtained in Proposition~\ref{prop:13}, i.e., 
	\begin{align*}
		&\int_\frac t2^t\left\|T_\Omega(t-\tau)\mathbb P\operatorname{div}(u(\tau)\otimes u(\tau))\right\|_{\dot H^m}\, d\tau
		\lesssim|\Omega|^{-1}
		M^{m, 2}_0(t)
		U^{m, q}_0(t)|\Omega|^{-\frac12(s-\frac12)}\left\|u_0\right\|_{\dot H^s}\\
		&=|\Omega|^{-\frac32}
		M^{m, 2}_1(t)
		(|\Omega|t)^\frac12
		U^{m, q}_0(t)|\Omega|^{-\frac12(s-\frac12)}\left\|u_0\right\|_{\dot H^s}, 
	\end{align*}
we have
	\begin{align*}
		&\int_\frac t2^t\left\|T_\Omega(t-\tau)\mathbb P\operatorname{div}(u(\tau)\otimes u(\tau))\right\|_{\dot H^m}\, d\tau
		\lesssim|\Omega|^{-\frac32}
		M^{m, 2}_1(t)
		U^{m, q}_0(t)
		\left(|\Omega|^{-\frac12(s-\frac12)}\left\|u_0\right\|_{\dot H^s}\right)^\frac12
		U^{s, q}_0(t)^\frac12.
	\end{align*}

Combining the above estimates, we have
	\begin{align*}
		U^{m,2}_1(t)
		\lesssim |\Omega|^\frac32\left\||x|u_0\right\|_{L^1}
		+\left(|\Omega|^\frac14\left\|u_0\right\|_{L^2}\right)^{\frac23}
		\big\|U^{0, 2}_0\big\|_{L^\infty}^\frac43
		+U^{m, q}_0(t)U^{s, q}_0(t)^{\frac12}\left(|\Omega|^{-\frac12(s-\frac12)}\left\|u_0\right\|_{\dot H^s}\right)^{\frac12}.
	\end{align*}
This completes the proof.
\qed

\subsection{$L^\infty$-type estimates}
\begin{proposition}\label{prop:15}
	Let
		$u$
	be the global solution obtained in Theorem~\ref{thm:1} and assume that the initial data
		$u_0$
	satisfies
		$u_0\in L^1(\mathbb R^3)^3\cap \dot H^s(\mathbb R^3)^3$
	and
		$|\Omega|^{-\frac12(s-\frac12)}\left\|u_0\right\|_{\dot H^s}
		\leq C_\star.$
	Then for any
		$m\geq 0$
	there exists a constant
		$C>0$
	depending only on
		$m, $
		$s, $
		$q, $
		$\theta, $
		$|\Omega|\left\|u_0\right\|_{L^1}, $
		$|\Omega|^\frac14\left\|u_0\right\|_{L^2}, $
	and
		$|\Omega|^{-\frac12(s-\frac12)}\left\|u_0\right\|_{\dot H^s}, $
	satisfying
		$U^{m, \infty}_0(t)\leq C$
	for all
		$t>0.$
\end{proposition}
\proof
For the linear part, as usual we have
	\begin{align*}
		\left\|T_\Omega(t)u_0\right\|_{\dot W^{s, \infty}}
		\lesssim M^{m, \infty}_0(t)\left\|u_0\right\|_{L^1}.
	\end{align*}
For the first half of the Duhamel integral, we have
	\begin{align}
		&\int_0^\frac t2\left\|T_\Omega(t-\tau)\mathbb P\operatorname{div}(u(\tau)\otimes u(\tau))\right\|_{\dot W^{m, \infty}}\, d\tau\notag\\
		&\lesssim M^{m, \infty}_{0}(t)t^{-\frac12}\int_0^\frac t2\left\|u(\tau)\right\|_{L^2}^2\, d\tau
		\lesssim M^{m, \infty}_{0}(t)t^{-\frac12}
		\left\|u_0\right\|_{L^2}^{\frac23}|\Omega|^{-\frac43}
		\big\|U^{0, 2}_0\big\|_{L^\infty}^\frac43\notag\\
		&= |\Omega|^{-1}M^{m, \infty}_0(t)(|\Omega|t)^{-\frac12}
		\left(|\Omega|^\frac14\left\|u_0\right\|_{L^2}\right)^\frac23
		\big\|U^{0, 2}_0\big\|_{L^\infty}^\frac43\label{eq:c}
	\end{align}
and
	\begin{align*}
		&\int_0^\frac t2\left\|T_\Omega(t-\tau)\mathbb P\operatorname{div}(u(\tau)\otimes u(\tau))\right\|_{\dot W^{m, \infty}}\, d\tau\\
		&\lesssim M^{m, \infty}_{0}(t)t^{-\frac12}\int_0^\frac t2\left\|u(\tau)\right\|_{L^2}^2\, d\tau
		\lesssim M^{m, \infty}_{0}(t)t^{\frac12}\left\|u_0\right\|_{L^2}^2\\
		&=|\Omega|^{-1}M^{m, \infty}_0(t)(|\Omega|t)^\frac12\left(|\Omega|^\frac14\left\|u_0\right\|_{L^2}\right)^2.
	\end{align*}
Therefore by interpolation we obtain
	\begin{align*}
		&\int_0^\frac t2\left\|T_\Omega(t-\tau)\mathbb P\operatorname{div}(u(\tau)\otimes u(\tau))\right\|_{\dot W^{m, \infty}}\, d\tau
		\lesssim |\Omega|^{-1}M^{m, \infty}_0(t)
		\left(|\Omega|^\frac14\left\|u_0\right\|_{L^2}\right)^\frac43
		\big\|U^{0, 2}_0\big\|_{L^\infty}^\frac23.
	\end{align*}
	
In order to deal with the second half of the Duhamel integral, we define the exponent
	$b\in(q, \infty)$
by the formula
	\begin{align*}
		\frac12=\frac1q-\frac s3+\frac1b.
	\end{align*}
The condition
	$b\in(q, \infty)$
can be verified by the inequalities
	$\frac1q-\frac s3<\frac1q<\frac12$
and
	$\frac2q-\frac s3\geq1-\frac1q>\frac12.$
Then we have
	\begin{align*}
		&\int_\frac t2^t\left\|T_\Omega(t-\tau)\mathbb P\operatorname{div}(u(\tau)\otimes u(\tau))\right\|_{\dot W^{m, \infty}}\, d\tau
		\lesssim \int_\frac t2^t(t-\tau)^{-\frac34}\left\|u(\tau)\otimes u(\tau)\right\|_{\dot H^{m+1}}\, d\tau\\
		&\lesssim \int_\frac t2^t(t-\tau)^{-\frac34}\left\|u(\tau)\right\|_{\dot W^{s, q}}\left\|u(\tau)\right\|_{\dot W^{m+1, b}}\, d\tau.
	\end{align*}
We may also take a real number
	$s_0\in (0, \frac3q)$
satisfying
	\begin{align*}
		\frac1b
		=\frac1q-\frac{s_0}3.
	\end{align*}
Then using the Sobolev embedding
	$\dot W^{s_0, q}(\mathbb R^3)
	\hookrightarrow L^b(\mathbb R^3), $
we have
	\begin{align*}
		&\int_\frac t2^t\left\|T_\Omega(t-\tau)\mathbb P\operatorname{div}(u(\tau)\otimes u(\tau))\right\|_{\dot W^{m, \infty}}\, d\tau
		\lesssim \int_\frac t2^t(t-\tau)^{-\frac34}\left\|u(\tau)\right\|_{\dot W^{s, q}}\left\|u(\tau)\right\|_{\dot W^{s_0+m+1, q}}\, d\tau\\
		&\lesssim |\Omega|^{-1}M^{s_0+m+1, q}_0(t)U^{s_0+m+1, q}_0(t)
		\int_\frac t2^t(t-\tau)^{-\frac34}\left\|u(\tau)\right\|_{\dot W^{s, q}}\, d\tau.
	\end{align*}
For the integral in the last term, we have
	\begin{align*}
		&\int_\frac t2^t(t-\tau)^{-\frac34}\left\|u(\tau)\right\|_{\dot W^{s, q}}\, d\tau
		\leq \left\|u\right\|_{L^\theta([0, \infty); \dot W^{s, q})}\left(
		\int_\frac t2^t(t-\tau)^{-\frac34\frac{\theta}{\theta-1}}\, d\tau
		\right)^{1-\frac1\theta}, 
	\end{align*}
and this is finite if and only if we have
	$\frac1\theta<\frac14.$
We show that the condition in fact holds in our setting. Indeed, since we have
	$\frac1\theta
	<\frac58-\frac3{2q}+\frac s4$, 
it suffices to show 
	$\frac58-\frac3{2q}+\frac s4\leq \frac14.$
This is in turn equivalent to
	$\frac14+\frac s6\leq\frac1q, $
and thus by \eqref{eq:5} it is sufficient to prove
	$\frac14+\frac s6
	\leq\frac13+\frac s9.$
This is equivalent to
	$s\leq\frac32, $
which is true. We therefore obtain
	\begin{align*}
		&\int_\frac t2^t(t-\tau)^{-\frac34}\left\|u(\tau)\right\|_{\dot W^{s, q}}\, d\tau
		\lesssim \left\|u\right\|_{L^\theta([0, \infty); \dot W^{s, q})}t^{\frac14-\frac1\theta}, 
	\end{align*}
and thus
	\begin{align*}
		&\int_\frac t2^t\left\|T_\Omega(t-\tau)\mathbb P\operatorname{div}(u(\tau)\otimes u(\tau))\right\|_{\dot W^{m, \infty}}\, d\tau\\
		&\lesssim |\Omega|^{-1}M^{s_0+m+1, q}_0(t)U^{s_0+m+1, q}_0(t)
		\left\|u\right\|_{L^\theta([0, \infty); \dot W^{s, q})}t^{\frac14-\frac1\theta}\\
		&\lesssim |\Omega|^{-1}U^{s_0+m+1, q}_0(t)t^{-\frac{s_0+m+1}2-\frac32(1-\frac1q)+\frac14-\frac1\theta}
		(1+|\Omega|t)^{-1+\frac2q}
		|\Omega|^{-\frac1\theta+\frac34-\frac3{2q}}\left\|u_0\right\|_{\dot H^s}\\
		&=|\Omega|^{-1}U^{s_0+m+1, q}_0(t)t^{-\frac m2-\frac32}(|\Omega|t)^{-\frac1\theta+\frac12-\frac3{2q}+\frac s2}
		(1+|\Omega|t)^{-1+\frac2q}
		|\Omega|^{-\frac12(s-\frac12)}\left\|u_0\right\|_{\dot H^s}\\
		&\leq|\Omega|^{-1}U^{s_0+m+1, q}_0(t)t^{-\frac m2-\frac32}
		|\Omega|^{-\frac12(s-\frac12)}\left\|u_0\right\|_{\dot H^s}.
	\end{align*}
We also have
	\begin{align}
		&\int_\frac t2^t\left\|T_\Omega(t-\tau)\mathbb P\operatorname{div}(u(\tau)\otimes u(\tau))\right\|_{\dot W^{m, \infty}}\, d\tau
		\lesssim \int_\frac t2^t(t-\tau)^{-\frac34}\left\|u(\tau)\right\|_{\dot W^{s, q}}\left\|u(\tau)\right\|_{\dot W^{s_0+m+1, q}}\, d\tau\notag\\
		&\lesssim |\Omega|^{-1}M^{s, q}_0(t)U^{s, q}_0(t)
		|\Omega|^{-1}M^{s_0+m+1, q}_0(t)U^{s_0+m+1, q}_0(t)
		t^{\frac14}\notag\\
		&= |\Omega|^{-2}U^{s, q}_0(t)U^{s_0+m+1, q}_0(t)
		t^{-\frac s2-\frac32(1-\frac1q)-\frac{s_0+m+1}2-\frac32(1-\frac1q)+\frac14}(1+|\Omega t|)^{-2+\frac4q}\notag\\
		&=|\Omega|^{-1}U^{s, q}_0(t)U^{s_0+m+1, q}_0(t)
		t^{-\frac m2-\frac32}(|\Omega|t)^{-1}(1+|\Omega t|)^{-2+\frac4q}\label{eq:d}\\
		&\leq|\Omega|^{-1}U^{s, q}_0(t)U^{s_0+m+1, q}_0(t)
		t^{-\frac m2-\frac32}(|\Omega|t)^{-1}.\notag
	\end{align}
By the above estimates, we obtain
	\begin{align*}
		&\int_\frac t2^t\left\|T_\Omega(t-\tau)\mathbb P\operatorname{div}(u(\tau)\otimes u(\tau))\right\|_{\dot W^{m, \infty}}\, d\tau\\
		&\lesssim |\Omega|^{-1}M^{m, \infty}_0(t)U^{s_0+m+1, q}_0(t)
		\left(U^{s, q}_0(t)+|\Omega|^{-\frac12(s-\frac12)}\left\|u_0\right\|_{\dot H^s}\right).
	\end{align*}
This completes the proof.
\qed

\begin{proposition}\label{prop:16}
	Let
		$u$
	be the global solution obtained in Theorem~\ref{thm:1} and assume that the initial data
		$u_0$
	satisfies
		$u_0\in L^1(\mathbb R^3)^3\cap \dot H^s(\mathbb R^3)^3, $
		$|x|u_0\in L^1(\mathbb R^3)^3, $
	and
		$|\Omega|^{-\frac12(s-\frac12)}\left\|u_0\right\|_{\dot H^s}
		\leq C_\star.$
	Then for any
		$m\geq 0$
	there exists a constant
		$C>0$
	depending only on
		$m, $
		$s, $
		$q, $
		$\theta, $
		$|\Omega|\left\|u_0\right\|_{L^1}, $
		$|\Omega|^\frac32\left\||x|u_0\right\|_{L^1}, $
		$|\Omega|^\frac14\left\|u_0\right\|_{L^2}, $
	and
		$|\Omega|^{-\frac12(s-\frac12)}\left\|u_0\right\|_{\dot H^s}, $
	satisfying
		$U^{m, \infty}_1(t)\leq C$
	for all
		$t>0.$
\end{proposition}
\proof
For the linear term, we routinely have
	\begin{align*}
		\left\|T_\Omega(t)u_0\right\|_{\dot W^{m, \infty}}
		\lesssim M^{m, \infty}_1(t)\left\||x|u_0\right\|_{L^1}.
	\end{align*}
For the first half of the Duhamel integral, by the argument in Proposition~\ref{prop:15}, we have
	\begin{align*}
		&\int_0^\frac t2\left\|T_\Omega(t-\tau)\mathbb P\operatorname{div}(u(\tau)\otimes u(\tau))\right\|_{\dot W^{m, \infty}}\, d\tau\\
		&\lesssim |\Omega|^{-1}M^{m, \infty}_0(t)(|\Omega|t)^{-\frac12}
		\left(|\Omega|^\frac14\left\|u_0\right\|_{L^2}\right)^\frac23
		\big\|U^{0, 2}_0\big\|_{L^\infty}^\frac43\\
		&=|\Omega|^{-\frac32}M^{m, \infty}_1(t)
		\left(|\Omega|^\frac14\left\|u_0\right\|_{L^2}\right)^\frac23
		\big\|U^{0, 2}_0\big\|_{L^\infty}^\frac43.
	\end{align*}

For the second half of the Duhamel term, we have
	\begin{align*}
		&\int_\frac t2^t\left\|T_\Omega(t-\tau)\mathbb P\operatorname{div}(u(\tau)\otimes u(\tau))\right\|_{\dot W^{m, \infty}}\, d\tau\\
		&\lesssim \int_\frac t2^t(t-\tau)^{-\frac34}\left\|u(\tau)\otimes u(\tau)\right\|_{\dot H^{m+1}}\, d\tau
		\lesssim \int_\frac t2^t(t-\tau)^{-\frac34}\left\|u(\tau)\right\|_{L^2} \left\|u(\tau)\right\|_{\dot W^{m+1, \infty}}\, d\tau\\
		&\lesssim |\Omega|^{-1}M^{0, 2}_0(t)U^{0, 2}_0(t)
		|\Omega|^{-1}M^{m+1, \infty}_0(t)U^{m+1, \infty}_0(t)
		t^\frac14\\
		&\leq |\Omega|^{-\frac32}M^{m, \infty}_1(t)
		(|\Omega|t)^{-\frac12}
		U^{0, 2}_0(t)
		U^{m+1, \infty}_0(t).
	\end{align*}
We also have
	\begin{align*}
		&\int_\frac t2^t\left\|T_\Omega(t-\tau)\mathbb P\operatorname{div}(u(\tau)\otimes u(\tau))\right\|_{\dot W^{m, \infty}}\, d\tau\\
		&\lesssim \int_\frac t2^t(t-\tau)^{-\frac34}\left\|u(\tau)\right\|_{L^2} \left\|u(\tau)\right\|_{\dot W^{m+1, \infty}}\, d\tau\\
		&\lesssim |\Omega|^{-\frac32}M^{m, \infty}_1(t)(|\Omega|t)^\frac14\left(|\Omega|^\frac14\left\|u_0\right\|_{L^2}\right)U^{m+1, \infty}_0(t).
	\end{align*}
By interpolation, we obtain
	\begin{align*}
		&\int_\frac t2^t\left\|T_\Omega(t-\tau)\mathbb P\operatorname{div}(u(\tau)\otimes u(\tau))\right\|_{\dot W^{m, \infty}}\, d\tau\\
		&\lesssim |\Omega|^{-\frac32}M^{m, \infty}_1(t)\left(|\Omega|^\frac14\left\|u_0\right\|_{L^2}\right)^\frac23
		U^{0, 2}_0(t)^\frac13
		U^{m+1, \infty}_0(t).
	\end{align*}
Combining the above estimates concludes the proof.
\qed

We finally prove Theorem~\ref{thm:3}.
\proof[Proof of Theorem~\ref{thm:3}]
The decay part follows by interpolating the estimates obtained in Section~\ref{sect:4}. For the asymptotics part, note that decay rates of the Duhamel term \eqref{eq:a} and \eqref{eq:b} in Proposition~\ref{prop:13}, and \eqref{eq:c} and \eqref{eq:d} in Proposition~\ref{prop:15}, are always strictly faster than the desired ones. Combining this with the estimate in Proposition~\ref{prop:2.4'}, we obtain the conclusion.
\qed

\subsection*{Acknowledgement}
The author would like to express his sincere gratitude to Professor Ryo Takada for his valuable comments and continuous encouragement.
This research was supported by FoPM, WINGS Program, the University of Tokyo.

\vspace{5pt}
\noindent{\bf{Data availability statement}} No data was used for the research described in the article.

\vspace{5pt}
\noindent{\bf{Declarations}}

\vspace{5pt}
\noindent{\bf Conflict of interest}
The author declares that he has no conflict of interest.

\bibliographystyle{alpha} 
\bibliography{refs}          

\end{document}